\documentclass[10pt,double,reqno]{amsart}

\usepackage{color}
\definecolor{r}{RGB}{0.9,0.3,0.1}
\definecolor{b}{rgb}{0.1,0.3,0.9}

\newtheorem{theo}{Theorem}[section]
\newtheorem{defin}[theo]{Definition}
\newtheorem{prop}[theo]{Proposition}

\newtheorem{lemm}[theo]{Lemma}
\newtheorem{rem}[theo]{Remark}

\numberwithin{equation}{section}
\newcommand{\al}{\alpha}
\newcommand{\be}{\beta}

\newcommand{\Ga}{\Gamma}
\newcommand{\la}{\lambda}

\newcommand{\om}{\omega}
\newcommand{\Om}{\Omega}

\newcommand{\si}{\sigma}

\newcommand{\ep}{\epsilon }
\newcommand{\te}{\theta}
\newcommand{\De}{\Delta}
\newcommand{\de}{\delta}

\newcommand{\pa}{\partial}

\newcommand{\R}{{\mathbb R}^n}

\newcommand{\ri}{\rightarrow}
\newcommand{\Rn}{{\mathbb R}^{n-1}}

\newcommand{\na}{\nabla}

\begin{document}

\title[Navier--Stokes equations with nonhomogeneous data]{Initial-Boundary value problem of the Navier-Stokes equations in the half space
with nonhomogeneous data}

\
\author{Tongkeun Chang}
\address{Department of Mathematics, Yonsei University \\
Seoul, 136-701, South Korea}
\email{chang7357@yonsei.ac.kr}

\author{Bum Ja Jin}
\address{Department of Mathematics, Mokpo National University, Muan-gun 534-729,  South Korea }
\email{bumjajin@mokpo.ac.kr}

\thanks{}

\begin{abstract}
This paper discusses the  solvability (global in time) of  the initial-boundary value problem of the Navier-stokes equations in the half space when
the initial data $ h\in  \dot{ B}_{q \si}^{\al-\frac{2}{q}}(\R_+)$
and the boundary data
 $ g\in  \dot{ B}_q^{\al-\frac{1}{q},\frac{\al}{2}-\frac{1}{2q}}(\Rn\times {\mathbb R}_+) $  with $g_n\in    \dot B^{\frac12 \al}_q ({\mathbb R}_+; \dot B^{-\frac1q}_q ({\mathbb
R}^{n-1}))\cap L^q({\mathbb R}_+;\dot{B}^{\al-\frac{1}{q}}(\Rn))$, for any $0<\al<2$ and $q =\frac{n+2}{\al+1}$. Compatibility condition \eqref{compatibility1} is required for  $h$ and $g$.\\



\end{abstract}

\maketitle

\section{\bf Introduction}
\setcounter{equation}{0}

Let $\R_+ = \{ x \in \R \, | \, x_n > 0 \}$, $n\geq 2$.
In this study, we consider the following nonstationary Navier--Stokes equations
\begin{align}\label{maineq2}
\begin{array}{l}\vspace{2mm}
u_t - \De u + \na p =-{\rm div}(u\otimes u), \qquad {\rm div} \, u =0 \mbox{ in }
 \R_+\times (0, \infty),\\
\hspace{30mm}u|_{t=0}= h, \qquad  u|_{x_n =0} = g,
\end{array}
\end{align}
where
 $u=(u_1,\cdots, u_n)$ and $p$ are the unknown velocity and pressure, respectively,
   $     h=(h_1,\cdots, h_n)$ is the given initial data, and $  g=(g_1,\cdots, g_n)$ is the given boundary data.

Abundant literature exists on Navier--Stokes equations with homogeneous boundary data ($g =0$)
%  with   homogeneous boundary data, that is, with $g=0$.
(See  \cite{amann,cannone,giga2,giga3,giga1,kozono1,sol1} and the references therein). %In particular,  for   the half space problem see \cite{amann,kozono1,sol1} and references therein.% for the problems
 %in other domains such as whole space, a bounded domain, or exterior domain.
%(with homogeneous or nonhomogeneous boundary data)..

Further, over the past decade, many mathematicians have focused on studying Navier--Stokes equations with nonhomogeneous boundary data ($g \neq 0$)
%
%There are  many  literatures  for the study of the
% Navier-Stokes equations in other domain such as whole space, a bounded domain, or exterior domain (with homogeneous or nonhomogeneous boundary data).
(See \cite{fernandes,amann1, amann2, CJ1, chang-jin, CJ2, farwig4,farwig6, farwig2, farwig3,grubb1, grubb2, grubb3, grubb,lewis,R, voss} and the references therein). %In particular,  for the  half space problem see  \cite{fernandes,amann1,amann2,lewis,voss}  and references therein.
 %for the problems in other domains such as whole space, a bounded domain, or exterior domain.

%V.A. Solonnikov\cite{sol1} showed the existence   of solution $u\in W^{2,1}_q({\mathbb R}^3_+\times (0,\infty)), q>\frac{5}{3}
%$  %for the homogeneous boundary data, that is,
% for the data small enough when  $h\in B^{2-\frac{2}{q}}_q(\R_+)$ and $g=0$.
The study closely relating to our present study is that by G. Grubb\cite{grubb3}, who used pseudo-differential operator techniques to realize the local in time existence of solution $u\in B^{\al,\frac{\al}{2}}_{q}(\Omega \times (0,T))$, $\infty>q>\frac{n+2}{\al+1}$ with $\al q>2$ in the interior or exterior domains, when $h\in \mbox{\r{B}}^{\al-\frac{2}{q}}_q(\Omega)$ and $g\in B^{\al-\frac{1}{q},\frac{\al}{2}-\frac{1}{2q}}_{q0}(\partial \Omega \times (0,T))$
 with $g_n=0$ (When $h=0$, the result in \cite{grubb3} was given up to the case $\al q>1$(and $\infty>q>\frac{n+2}{\al+1})$. See also \cite{grubb1,grubb2,grubb,sol1}). Here, let $\mbox{\r{B}}^s_q(S)$ ($S\subset {\mathbb R}^m$) be the set of distributions in Besov space $B^s_q({\mathbb R}^m)$ supported in $\bar{S}$, and $B^{s,\frac{s}{2}}_{q0}(S\times (0, T))$ be the set of distributions in anisotropic Besov space $B^{s,\frac{s}{2}}_{q}({\mathbb R}^m\times (-\infty,T])$ supported in $\bar{S} \times [0,T]$.

 In Refs. \cite{amann1,amann2,farwig4, farwig6,farwig2,farwig3}, rough initial and boundary data were considered for the local data in the time existence of weak or very weak solutions.
 % (See also  \cite{ farwig3,farwig4, farwig6}).
  In Refs. \cite{fernandes,lewis,voss}, a mild-type solution was considered in the half space when the rough initial and boundary data are given.
Recently, Chang and Jin \cite{chang-jin} studied the local in time solvability of Navier--Stokes equations when $h\in B^{-\frac{2}{q}}_q(\R_+)$ and $g\in B^{-\frac{1}{q},-\frac{1}{2q}}_q(\Rn\times {\mathbb R}_+)$, $q>n+2$. By the same authors  \cite{CJ2}, this result is extended  to  global time existence with small initial and boundary data.

%{\r
This study aims to extend the result of \cite{grubb3} to any $h\in  \dot B^{\al-\frac{2}{q}}_{q \si}(\R_+)$ and $g\in \dot B^{\al-\frac{1}{q},\frac{\al}{2}-\frac{1}{2q}}_q(\Rn\times (0,\infty))$ with $g_n\in  B^{\frac{\al}{2}}_q(0,\infty;\dot{B}^{-\frac{1}{q}}_q(\Rn))\cap L^q(0,\infty;\dot{B}^{\al-\frac{1}{q}}(\Rn))$, where $q=\frac{n+2}{\al+1}$ and $0<\al<2$.
 For $\al > \frac3q$, the following compatibility condition is required:
 \begin{align}\label{compatibility}
 g|_{t =0} = h|_{x_n=0} \quad \mbox{on} \quad \Rn.
 \end{align}
The compatibility \eqref{compatibility} can be generalized to  any $\al>0$ and $q>1$ as follows:\\
\begin{equation}
\label{compatibility1}
g-\Gamma_t *\tilde{h}|_{x_n=0} \in
 %$(h,g)\in
\dot { B}^{\alpha-\frac{1}{q},\frac{\alpha}{2}-\frac{1}{2q}}_{q(0)}(\Rn\times (0,\infty)),
\end{equation}%=\mbox{
where $B^{s,\frac{s}{2}}_{q(0)}(\Rn\times (0,\infty))$ is  the completion of $C^\infty_0(\Rn\times (0,\infty))$ in ${B}^{s,\frac{s}{2}}_{q}(\Rn\times (0,\infty))$,
 $\tilde{h}\in \dot B^{\al-\frac{2}{q}}_{q\sigma}(\R)$ is some solenoidal extension of $h$ to $\R$  with
$\|\tilde{h}\|_{ \dot B^{\al-\frac{2}{q}}_{q}(\R)} \approx \|h\|_{\dot B^{\al-\frac{2}{q}}_{q}(\R_+)}$
and $\Gamma_t * f|_{x_n=0}:=\int_{\R}\Gamma(x'-y',y_n,t)f(y)dy.$
(According to Lemma \ref{proheat1}, $\Gamma_t * f|_{x_n=0}\in \dot { B}^{\alpha-\frac{1}{q},\frac{\alpha}{2}-\frac{1}{2q}}_{q}(\Rn\times (0,\infty)) $ for any $\alpha>0, \, q>1$ when  $f\in \dot B^{\al-\frac{2}{q}}_{q\sigma}(\R)$.)

The following text states our main result.
\begin{theo}
\label{thm3}
Let $0 <\al <2$ and $1 < q=\frac{n+2}{\al+1} < \infty$. Further, let $h\in \dot  { B}_{q\si}^{\al-\frac{2}{q}}(\R_+)$ and $g \in  \dot{B}^{\al -\frac1q, \frac{\al}2 - \frac1{2q}}_q (\Rn \times {\mathbb R}_+)$ with $g_n \in    \dot{B}^{\frac{\al}2 }_q (0,\infty; \dot B^{-\frac1q}_q ({\mathbb
R}^{n-1}))\cap  L^q(0,\infty;\dot{B}^{\al-\frac{1}{q}}_q(\Rn))$. In addition,
 we assume that $(h,g)$  satisfies the generalized compatibility condition \eqref{compatibility1}. %, that is, $g-\Gamma*_th|_{x_n=0} \in
% %$(h,g)\in
% \dot{ B}^{\alpha-\frac{1}{q},\frac{\alpha}{2}-\frac{1}{2q}}_{q(0)}(\Rn\times (0,\infty))$.
Then, there exists $\ep^* > 0$ such that if
 \begin{align*}
 \|h\|_{   \dot{ B}^{\al-\frac{2}{q}}_{q }({\mathbb
R}^{n}_+)}+
\|g\|_{ \dot{ B}^{\al-\frac{1}{q},\frac\al 2-\frac{1}{2q}}_{q}({\mathbb
R}^{n-1} \times (0,\infty))}
 +\|g_n\|_{   \dot{B}^{\frac{\al}2 }_q (0,\infty; \dot B^{-\frac1q}_q ({\mathbb
R}^{n-1}))}\\
+\|g_n\|_{ L^q(0,\infty;\dot{B}^{\al-\frac{1}{q}}_q(\Rn))} \leq \ep^*,
\end{align*}
 then the   \eqref{maineq2} has  a unique weak solution
$u\in \dot B^{\al,\frac\al 2 }_{q}({\mathbb R}^n_+\times (0,\infty))$.
\end{theo}
Section \ref{notation} further explains the spaces and notations.

Note that as the nonstationary Navier--Stokes equations is invariant under the scaling,
\begin{align*}
u_\la(x,t) = \la u(\la x, \la^2 t),\quad   p_\la (x,t) = \la^2 p(\la x, \la^2t), \\
h_{\la} = \la h (\la x),  \quad  g_\la(x,t) = \la g(\la x, \la^2 t), \,\, \la > 0,
\end{align*}
significantly considering \eqref{maineq2} in the so-called critical spaces, i.e., the function space invariant under the scaling $u(\cdot ) \ri
\la u(\la λ\cdot )$ is very important. In Theorem \ref{thm3}, the function spaces containing solutions are critical spaces.

%From (5) of Proposition \ref{prop2} and Remark \ref{rem0208}, we have
%\begin{align*}
%\| u(t)\|_{\dot B^{\al -\frac2q}_q(\R_+)} \leq c \| u \|_{\dot B^{\al,\frac{\al}2}_q (\R_+ \times {\mathbb R}_+)} \qquad t \in {\mathbb R}_+, \quad \al > \frac2q,
%\end{align*}
%where $c>0$ is not dependent of $t$ and $u$. Hence, we obtain
%\begin{coro}
%If $\al > \frac2q$, then the solution obtained in Theorem \ref{thm3} is in $L^\infty (0, \infty; \dot B^{\al -\frac2q}_q(\R_+))$.
%\end{coro}

For the proof of Theorem \ref{thm3}, it is necessary to study the initial-boundary value problem of the Stokes equations in $\R_+\times (0,\infty)$ as follows:
\begin{align}\label{maineq-stokes}
\begin{array}{l}\vspace{2mm}
u_t - \De u + \na p =f, \qquad {\rm div} \, u =0 \mbox{ in }
 \R_+\times (0,\infty),\\
\hspace{30mm}u|_{t=0}= h, \qquad  u|_{x_n =0} = g.
\end{array}
\end{align}

Various studies have been conducted on the solvability of the Stokes equations \eqref{maineq-stokes}
with homogeneous or nonhomogeneous boundary data. Refs. \cite{cannone,giga,giga1,KS2,kozono1,sol1,Sol-2} and the references can be referred to for Stokes problems with homogeneous boundary data, whereas Refs.~\cite{grubb1, grubb2, grubb3,grubb,KS1,KS2,sol1,Sol-1, raymond1} and the references therein can be referred to for Stokes problems with nonhomogeneous boundary data.

Koch and Solonnikov~\cite{KS1} showed
 the unique local in time existence of solution $u \in W^{1,\frac12}_{p}( \Om
\times (0,T))$  of Stokes equations in a bounded  convex domain $\Omega$ with $C^2$ boundaries % with comparable results to the one in \cite{KS2}
when $f=\mbox{div}{\mathcal F}$, ${\mathcal F}\in L^p(\Om \times (0,T))$, $h=0$, and $g\in  W^{1-\frac{1}{p},\frac{1}{2}-\frac{1}{2p}}_{p0}(\pa \Om
\times (0,T))$ with  $g_n=\mbox{div}{\mathcal A}$ and ${\mathcal A}\in W^{\frac{1}{2}}_{p0}(0,T;W^{1-\frac{1}{p}}_p(\pa \Om ))$. %  are given.
 See also \cite{KS2,sol1,Sol-2}.

In \cite{grubb3}, an interior or exterior domain problem was considered for $h\in \mbox{\r{B}}^{\al-\frac{2}{q}}_q(\Omega)$ and $g\in B^{\al-\frac{1}{q},\frac{\al}{2}-\frac{1}{2q}}_{q0}(\Omega \times (0, T))$, with $g_n=0$ (The result was given for $\al q>1$ if $h=0$ and for $\al q>2$ if $h\neq 0$. See also \cite{grubb1,grubb2,grubb,Sol-1}).
In \cite{farwig2,raymond1}, Stokes equations were solved for rough data including distributions.

The following theorem states our result on the unique solvability of the Stokes equations \eqref{maineq-stokes}.
\begin{theo}
\label{thm-stokes}
Let $1 < q < \infty$, $0 <\al <2$. In addition, let $h$ and $g$ be the same as those given in Theorem \ref{thm3}, and $f={\rm div } \, {\mathcal F}.$ Assume that ${\mathcal F}\in L^p (0,\infty; B^{\be}_{p}(\R_+ ))$ for some $(\be, p)$ satisfying conditions $p\leq q$, $0<\be<\al\leq \be+1<2$, $0=1-\al+\be-(n+2)(\frac{1}{p}-\frac{1}{q})$, and $\frac{n+1}p > \frac{n+2}q -\al$. Then, there is a weak solution $u\in \dot B^{\al,\frac\al 2 }_{q}({\mathbb R}^n_+\times (0,\infty))$ satisfying
\begin{align*}
\| u\|_{ \dot B^{\al,\frac\al 2 }_{q}({\mathbb R}^n_+\times (0,\infty))}
 &\leq c\Big( \|h\|_{  { \dot B}^{\al-\frac{2}{q}}_{q\si}({\mathbb
R}^{n}_+)}
+\|g \|_{ \dot B^{\al-\frac{1}{q},\frac\al 2-\frac{1}{2q}}_{q}({\mathbb
R}^{n-1} \times (0,\infty))} +\|g_n\|_{   {\dot B}^{\frac12 \al
}_{q} (0,\infty; \dot B^{-\frac1q}_q ({\mathbb
R}^{n-1}))} \\
  &\qquad +\|g_n\|_{ L^q(0,\infty;\dot{B}^{\al-\frac{1}{q}}_{q}(\Rn))}
+\|{\mathcal F}\|_{ L^p (0, \infty; \dot B^{\be}_{p}(\R_+ ))} \Big).
\end{align*}

Moreover, if $\al-\frac{n+2}{q}=-\frac{n+2}{r}$ for some $r$ with $1<r<\infty$, then the solution is unique in the class $\dot{B}^{\al,\frac{\al}{2}}_q(\R_+\times (0,\infty))$.
%In particular, if $\al q\geq 2$, then the condition  for ${\mathcal F}$ could be replaced by
%      ${\mathcal F}\in \dot {B}^{\be,\frac{\be}2}_{p_1}(\R_+ \times {\mathbb R}_+)$ for some $(\be, p_1,p_2)$ with the condition that $p_1,p_2\leq q$, $\be<\al<\be+1$ and $1+\be=\al+\frac{n}{p_1}+\frac{2}{p_2}-\frac{n+2}{q}$.
\end{theo}

The remainder of this paper is organized as follows.
In Section \ref{notation}, we introduce the notations and function spaces.
Section \ref{preliminary} presents the preliminary estimates in homogeneous anisotropic Besov spaces for the Riesz and Poisson heat operators. %In Section \ref{initial-zeroboundary}, we consider the Stokes equations \eqref{maineq-stokes} with zero force and zero boundary data.
In Section \ref{zero}, we consider Stokes equations \eqref{maineq-stokes} with zero force and zero initial velocity, and provide proof of Theorem \ref{Rn-1}.
Sections  \ref{general} show the proof of Theorem \ref{thm-stokes} with the help of Theorem \ref{Rn-1} and the preliminary estimates in Section \ref{preliminary}.
In Section \ref{nonlinear}, we give the proof of Theorem \ref{thm3} by constructing approximate solutions.

Our arguments in this paper are based on the elementary estimates of heat and Laplace operators.

\section{Notations and Definitions}

\label{notation}
\setcounter{equation}{0}
The points of spaces $\Rn$ and $\R$ are denoted by $x'$ and $x=(x',x_n)$, respectively.
%We denote $x  = (x',x_n) \in {\mathbb R}^n_+$ for $x'\in \Rn $ and denote
In addition, multiple derivatives are denoted by $ D^{k}_x D^{m}_t = \frac{\pa^{|k|}}{\pa x^{k}} \frac{\pa^{m} }{\pa t}$ for multi-index
$ k$ and nonnegative integer $ m$.
For vector field $f=(f_1,\cdots, f_n)$ on $\R$, we write $f'=(f_1,\cdots, f_{n-1})$ and $f=(f',f_n)$.
Throughout this paper, we denote various generic constants by using $c$.
Let $\R_+=\{x=(x',x_n): x_n>0\}$, $\overline{\R_+}=\{x=(x',x_n): x_n\geq 0\}$, and
${\mathbb R}_+=(0,\infty)$.

For the Banach space $X$ and interval $I$, we denote by  $X'$ the dual space of $X$, and by $L^p(I;X), 1\leq p\leq \infty$  the usual Bochner space.
For $0< \theta<1$ and $1<p<\infty$,  denote by   $(X,Y)_{\theta,p}$  the real interpolation  space of the Banach space $X$ and $Y$.
For $1\leq p\leq \infty$, we write $p'=\frac{p}{p-1}$. For $s\in {\mathbb R}$, we write  $[s]=$ the largest integer less than $s$.

 Let   $\Omega$ be a $m$-dimensional Lipschitz domain, $m\geq 1$. %,% and let
 Let $1\leq p\leq \infty$ and $k$ be a nonnegative integer.
 The norms of usual  Lebesque space $L^p(\Omega)$, the usual homogeneous  Sobolev space $\dot W^k_p(\Omega)$   are written by $\|\cdot\|_{L^p(\Omega)}, \ \|\cdot\|_{\dot{W}^k_p(\Omega)}$, respectively.
    Note that $ \dot{W}^0_p(\Omega)=L^p(\Omega)$.

   For $s\in {\mathbb R}$, we denote by $\dot B^s_{p,q} ({\mathbb R}^m), 1\leq p,q\leq \infty$  the  usual homogeneous Besov space in ${\mathbb R}^m$ and
 denote by   $\dot{B}^s_{p,q}(\Omega)$  %the usual Besov spaces and the homogeneous Besov spaces defined by
 the restriction of  $\dot{B}^s_{p,q}({\mathbb R}^m)$ to $\Omega$.
 For the simplicity, set $  \ \dot B^s_{p}(\Omega)=\dot B^s_{p,p}(\Omega)$.

 It is known that  %$B^s_p(\Omega)=L^p(\Omega)\cap \dot{B}^s_p(\Omega)$ for $s>0$; $B^s_p(\Omega)=L^p(\Omega)+\dot{B}^s_{p}(\Omega)$ for $s<0$;
  $\dot B^{s}_p(\Omega)=(L^p(\Omega), \dot W^{k}_p(\Omega))_{\frac{s}{k}, p}$ for $0<s<k$ and
 $\dot B^{s}_p(\Omega)=(\dot B^{s_1}_p(\Omega), \dot B^{s_2}_p(\Omega))_{\theta, p}$
 for $s=(1-\theta)s_1+\theta s_2$, $0<\theta<1$  and $1<p<\infty$.
 In particular,
$\dot B^{s}_p(\Omega)= \Big(\dot B^{-s}_{p'}(\Omega)\Big)'
$ if $-1+\frac{1}{p}<s<\frac{1}{p}$ and $1<p<\infty$.    See \cite{Tr} for the reference.

Denote by $\dot B^s_{q\sigma}(\R) = \{ f \in \dot B^s_q (\R) \, | \, {\rm div} \, f =0 \}$   and $\dot B^s_{q\sigma}(\Omega)$ is the restriction of $\dot B^s_{q\sigma}(\R)$ to $\Omega$.
%The nonhomogeneous space $ B^s_{q\sigma}(\R)$ and $ B^s_{q\sigma}(\Omega)$ are defined similarly.

Now, we introduce homogeneous anisotropic  Besov space and its properties (See %chapter 4 of \cite{Triebel},
Chapter 5 of \cite{Triebel2}, and Chapter 3 of \cite{amman-anisotropic} for the definition of homogeneous anisotropic  spaces and their properties, although different notations were used in each books).

Define  homogeneous anisotropic   Besov space $\dot B^{s,\frac{s}{2}}_p({\mathbb R}^n\times {\mathbb R})$ by
\begin{align*}
\dot B^{s,\frac{s}{2}}_p({\mathbb R}^n\times {\mathbb R})=\left\{\begin{array}{l} \vspace{2mm}
L^p({\mathbb R};\dot B^s_p({\mathbb R}^n))\cap L^p({\mathbb R}^n;\dot B^{\frac{s}{2}}_p({\mathbb R}))\mbox{ if }s>0,\\ \vspace{2mm}
L^p({\mathbb R};\dot B^s_p({\mathbb R}^n))+ L^p({\mathbb R}^n;\dot B^{\frac{s}{2}}_p({\mathbb R}))\mbox{ if }s<0,\\
(\dot B^{-\ep,-\frac{\ep}{2}}_p({\mathbb R}^n\times {\mathbb R}),\dot B^{\ep,\frac{\ep}{2}}_p({\mathbb R}^n\times {\mathbb R}))_{\frac{1}{2},p},  \ep>0\mbox{   if   }s=0.\end{array}\right.
\end{align*}
The above  definition is equivalent to the definitions in \cite{amman-anisotropic,Triebel2}.
Denote by $\dot {B}^{s, \frac s2 }_{q} (\Omega\times I)$
%are defined by
 the restriction of $\dot {B}^{s, \frac s2 }_{q} ({\mathbb R}^n\times {\mathbb R})$
 to $\Omega \times I$, with norm
\begin{align*}
%\|f\|_{ B^{s, \frac{s}2 }_p (\Omega \times I)}=\inf\{ \|F\|_{ B^{s, \frac{s}2 }_p(\R \times {\mathbb R})}: F\in  B^{s, \frac{s}2}_p(\R \times {\mathbb R})\mbox{ with } F|_{\Omega \times I}=f\},\\
\|f\|_{\dot B^{s, \frac{s}2 }_p (\Omega \times I)}=\inf\{ \|F\|_{\dot B^{s, \frac{s}2 }_p(\R \times {\mathbb R})}: F\in \dot B^{s, \frac{s}2}_p(\R \times {\mathbb R})\mbox{ with } F|_{\Omega \times I}=f\}.
\end{align*}

For $ k\in {\mathbb N}\cup\{0\}$, denote by  $\dot{W}^{2k,k}_q(\Omega \times I)$  the usual homogeneous anisotropic Sobolev space.

  The  properties of the homogeneous  anisotropic Besov spaces are comparable with the properties of  Besov spaces:
%?????????????
%
%The following properties can be shown through the same argument as that for
%the usual Sobolev and Besov spaces in $\R$.
In particular, the following properties can be used in this paper.
\begin{prop}
\label{prop2}
\begin{itemize}
\item[(1)]
 The real interpolation method gives
%\begin{align*}
%( \dot  L^{p_1}(\R \times {\mathbb R}), \dot W^{2,1}_{p_0}(\R \times {\mathbb R}) )_{\te, p} =
%\dot B^{2\te,\te}_{p}(\R \times {\mathbb R}),\\
%(\dot B_{p}^{\al_0,\frac{ \al_0}2 }(\R \times {\mathbb R}),\dot B_{p}^{\al_1,\frac{\al_1}2}(\R \times {\mathbb R}) )_{\te,
%p} = \dot B^{\al,\frac{\al}2}_{p}(\R \times {\mathbb R})
%\end{align*}
%for $ 1 \leq p_0, p_1, p \leq  \infty $ and $\al_0, \al_1 \in {\mathbb R}$ satisfying $ \frac1p =\frac{1-\te}{p_0} +\frac{\te}{p_1}$ and  $ \al = \al_0(1-\te)+\te \al_1, \,  0 < \te < 1$, respectively.

\begin{align*}
\dot B^{s,\frac{s}{2}}_{p}(\R\times {\mathbb R} )   &=(L^p(\R\times {\mathbb R}),
    \dot W^{2k, k}_{p}(\R\times {\mathbb R}))_{\frac{s}{2k},p},\  0<s<2k,\\
    \dot B^{s,\frac{s}{2}}_{p}(\R\times {\mathbb R})   &=(\dot B^{s_1,\frac{s_1}{2}}_{p}(\R\times {\mathbb R}),
   \dot B^{s_2,\frac{s_2}{2}}_{p}(\R\times {\mathbb R}))_{\te,p},\ \, 0<\te<1, \ s = (1-\te)s_1 + \te s_2,\ s_1< s_2
   \end{align*}  for any real number  $1<p<\infty$.

%\item[(2)]
%
%\begin{eqnarray*}
% (\dot B_{p}^{\al_0,\frac{ \al_0}2 }(\R \times {\mathbb R}),\dot B_{p}^{\al_1,\frac{\al_1}2}(\R \times {\mathbb R}) )_{\te,
%p} = \dot B^{\al,\frac{\al}2}_{p}(\R \times {\mathbb R})
%\end{eqnarray*}
%for $ 1 \leq p \leq \infty, \, \, \al = (1-\te) \al_0 + \te \al_1,
%0 < \te < 1$.

\item[(2)] For $s > 0$
%For $0 < \al < 2$, $\dot B^{\al, \frac{\al}2}_p (\R \times {\mathbb R}) = L^p ({\mathbb R}; \dot B^\al_p (\R)) \cap L^p (\R ; \dot B^{\frac{\al}2}_p ({\mathbb R}))$ with
%\begin{align}\label{besovnorm1}
%\| f\|_{\dot B^{\al, \frac{\al}2}_p (\R \times {\mathbb R})} \approx \|f\|_{L^p({\mathbb R}; \dot B^\al_p (\R )) }
%     + \| f\|_{L^p (\R ; \dot B^{\frac{\al}2}_p ({\mathbb R}))}.
%\end{align}
\begin{align*}
\dot B^{s,\frac{s}{2}}_{p}(\Omega\times I)   &=L^p(I;\dot B^s_p(\Omega))\cap L^p(\Omega; \dot B^{\frac{s}{2}}_{p}(I)).%\\
% B^{s,\frac{s}{2}}_{p}(\Omega\times I)   &=L^p(\Omega\times I)\cap \dot B^{s,\frac{s}{2}}_{p}(\Omega\times I)
 \end{align*}
% and for $s < 0$
% \begin{align*}
%\dot B^{s,\frac{s}{2}}_{p}(\Omega\times I) & =L^p(I;\dot B^s_p(\Omega))+ L^p(\Omega; \dot B^{\frac{s}{2}}_{p}(I)),%\\
%% B^{s,\frac{s}{2}}_{p}(\Omega\times I)   &=L^p(\Omega\times I)+ \dot B^{s,\frac{s}{2}}_{p}(\Omega\times I).
% \end{align*}

\item[(3)]
Let $ 1 < p_0\leq p_1  < \infty, \,  1 < q_0\leq q_1  < \infty$ and $ s_0\geq s_1$ with $s_0 - \frac{n+2}{p_0} = s_1 - \frac{n+2}{p_1}$.
Then, the following inclusions hold
\begin{align*}
%&\dot W^{s_0,\frac{s_0}{2}}_{p_0} (\R \times {\mathbb R}) \subset  \dot B^{s_1,\frac{s_1}{2}}_{p_1} (\R \times {\mathbb R}), \qquad
\dot  B^{s_0,\frac{s_0}{2}}_{p_0 q_0} (\R \times {\mathbb R}) \subset   \dot W^{s_1,\frac{s_1}{2}}_{p_1 q_1} (\R \times {\mathbb R}),\qquad
\dot  B^{s_0,\frac{s_0}{2}}_{p_0 q_0} (\R \times {\mathbb R}) \subset   \dot B^{s_1,\frac{s_1}{2}}_{p_1 q_1} (\R \times {\mathbb R}).
\end{align*}
%
%If $ 1 <  p_0\leq p_1 < \infty$ and $ s_0\geq s_1$ with $s_0 - \frac{n+2}{p_0} = s_1 - \frac{n+2}{p_1}$.
%Then, the following inclusion holds
%\begin{align*}
%%\label{prop3}
%\dot W^{s_0,\frac{s_0}{2}}_{p_0} (\R \times {\mathbb R}) \subset  \dot B^{s_1,\frac{s_1}{2}}_{p_1} (\R \times {\mathbb R}), \\
%%\mbox{ with }
%%\|f \|_{\dot H^{s_1,\frac{s_1}{2}}_{p_1} (\R \times {\mathbb R})   } \leq c \| f \|_{ \dot H^{s_0,\frac{s_0}{2}}_{p_0} (\R \times {\mathbb R})  },\\
%\dot B^{s_0,\frac{s_0}{2}}_{p_0} (\R \times {\mathbb R}) \subset  \dot W^{s_1,\frac{s_1}{2}}_{p_1 } (\R \times {\mathbb R}).
%%\mbox{ with }
%%\label{prop3}
%%\|f \|_{\dot B^{s_1,\frac{s_1}{2}}_{p_1 q_1 } (\R \times {\mathbb R})   } \leq c \| f \|_{ \dot B^{s_0,\frac{s_0}{2}}_{p_0 q_0} (\R \times {\mathbb R})  }.
%\end{align*}

\item[(4)]
For $f \in  \dot W_p^{\al, \frac{\al}2} (\R \times {\mathbb R})$ and $f \in  \dot B_p^{\al, \frac{\al}2} (\R \times {\mathbb R}), \, \al > \frac1p$, $f|_{x_n =0} \in  \dot B_p^{\al -\frac1p, \frac{\al}2 -\frac1{2p}}(\Rn \times {\mathbb R})$ with
\begin{align*}
\| f \|_{ \dot B_p^{\al -\frac1p, \frac{\al}2 -\frac1{2p}} (\Rn \times {\mathbb R})} \leq c \| f \|_{ \dot W_p^{\al, \frac{\al}2} (\R \times {\mathbb R})},\\
\| f \|_{\dot  B_p^{\al -\frac1p, \frac{\al}2 -\frac1{2p}} (\Rn \times {\mathbb R})} \leq c \| f \|_{ \dot B_p^{\al, \frac{\al}2} (\R \times {\mathbb R})}.
\end{align*}

\item[(5)]
For $f \in  \dot W_p^{\al, \frac{\al}2} (\R \times {\mathbb R})$ and $f \in  \dot B_p^{\al, \frac{\al}2} (\R \times {\mathbb R}), \, \al > \frac2p$, $f|_{t =0 } \in  \dot B_p^{\al -\frac2p }(\R )$ with
\begin{align*}
\| f|_{t =0} \|_{ \dot B_p^{\al -\frac2p } (\R )} \leq c \| f \|_{ \dot W_p^{\al, \frac{\al}2} (\R \times {\mathbb R})},\quad
\| f|_{t=0} \|_{ \dot B_p^{\al -\frac2p } (\R )} \leq c \| f \|_{ \dot B_p^{\al, \frac{\al}2} (\R \times {\mathbb R})}.
\end{align*}

 \end{itemize}
\end{prop}
For the proof of $(1)$, refer to page 169 in \cite{BL} or (a) of Theorem 2.4.2.1 in \cite{Tr}, and  Theorem 6.4.5 in \cite{BL}.
For the proof of $(2)$, refer to the proof of  Theorem 3 in \cite{DT}.
%For the details of the proof of (2), we refer
%\cite{BL} (in particular Definition 6.2.2, Theorem 6.2.4
%and Theorem 6.4.5 in \cite{BL}) for $(3)$.
For the proof of $(3)$, refer to the proof of Theorem 6.5.1 in \cite{BL}, and for the proofs of (4) and (5), refer to Theorem 6.6.1 in \cite{BL}.

\begin{rem}\label{rem0208}
The properties in Proposition \ref{prop2} of the homogeneous anisotropic Besov spaces in $\R \times {\mathbb R}$ hold for the homogeneous Besov spaces in $\R_+ \times  {\mathbb R}_+$ and $\Rn \times  {\mathbb R}_+ $ (see \cite{amman-anisotropic, Triebel,Triebel2}).

\end{rem}

Denote by $\dot B^{s,\frac{s}{2}}_{p(0)}(\Rn \times {\mathbb R}_+)$ the completion of $C^\infty_0(\Rn\times {\mathbb R}_+)$ in $\dot B^{s,\frac{s}{2}}_{p}(\Rn \times {\mathbb R}_+)$.
%
%
%$B^{s,\frac{s}{2}}_{p0}(\Omega \times (0,T))$ the set of distributions  $f\in  B^{s,\frac{s}{2}}_p(\Omega \times (-\infty,T))$ supported in $  \Omega \times [0,T)$. %with a finite norm
% $\|f\|_{\dot B^{s,\frac{s}{2}}_{p0}(\R_+ \times ({\mathbb R}_+))}=\|f\|_{\dot B^{s,\frac{s}{2}}_{p}({\mathbb R}^n \times {\mathbb R})}  $.
It is known that
 \begin{align}\label{homo0221}
 \dot  B^{s,\frac{s}{2}}_{p(0)}(\Rn \times {\mathbb R}_+) & = \left\{\begin{array}{l} \vspace{2mm}
 \{g\in \dot B^{s,\frac{s}{2}}_{p}(\Rn \times  {\mathbb R}_+): g|_{t=0}=0\}\quad \mbox{ if } \quad  2> s > \frac{2}{p},\\
  \dot B^{s,\frac{s}{2}}_{p}(\Rn \times {\mathbb R}_+)
 \quad \mbox{ if } \quad 0\leq s<\frac{2}{p},\\
 \Big( \dot B^{-s,-\frac{s}{2}}_{p'}(\Rn \times {\mathbb R}_+)\Big)'
 \quad\mbox{ if }-2<s<0,
 \end{array}\right.
 \end{align}
where $p'$ is the H$\ddot{o}$lder conjugate of $p$, that is, $\frac1p + \frac1{p'} =1$.

\begin{rem}\label{rem0101-1}
\begin{itemize}
\item[(1)]
Let $f\in \dot B^{s,\frac{s}{2}}_{p(0)}(\Rn\times {\mathbb R}_+)$.
If $ 0 \leq  s $, then its zero extension $\tilde{f}$ to $\Rn \times {\mathbb R}$ is contained in $\dot B^{s,\frac{s}{2}}_{p}(\Rn \times {\mathbb R})$ such that $\| \tilde{f}\|_{\dot B^{s,\frac{s}{2}}_{p}(\Rn \times {\mathbb R})}\approx \|f\|_{\dot B^{s,\frac{s}{2}}_{p}(\Rn \times {\mathbb R}_+)}$. If $s < 0$, then the zero extension $\tilde f$ of the distribution $f$ is defined by
\begin{align*}
<< \tilde f, \phi>>
:= < f, \phi|_{\Rn \times {\mathbb R}_+} >,
\end{align*}
for any $\Phi\in \dot B^{-s, -\frac{s}2}_{p'} (\Rn \times {\mathbb R} )$, where $p'$ is the H$\ddot{o}$lder conjugate of $p$, that is, $\frac1p + \frac1{p'} =1$. Then, $\| \tilde{f}\|_{\dot B^{s,\frac{s}{2}}_{p}({\mathbb R}^n\times {\mathbb R})}\approx \|f\|_{\dot B^{s,\frac{s}{2}}_{p}({\mathbb R}^n \times  {\mathbb R}_+)}$.
Here $ <\cdot , \cdot >$ is the duality pairing between $ \dot B^{s, \frac{s}2}_{p(0)} (\Rn \times  {\mathbb R}_+)$ and $ \dot B^{-s, -\frac{s}2}_{p'} (\Rn \times {\mathbb R}_+)  $.

\item[(2)]

Let $f\in \dot B^{s,\frac{s}{2}}_{p}({\mathbb R}^n_+ \times {\mathbb R}_+)$, $ 0 <   s  <2 $. Let $\tilde{f} $ be Adam's extension of $f$ with respect to space (see Theorem 5.19 in \cite{AF}) and   reflective extension with respect to time. Then, $\tilde{f} $ is in $\dot B^{s,\frac{s}{2}}_{p}({\mathbb R}^n\times {\mathbb R})$ with  $\| \tilde{f}\|_{\dot B^{s,\frac{s}{2}}_{p}({\mathbb R}^n\times {\mathbb R})}\leq c \|f\|_{\dot B^{s,\frac{s}{2}}_{p}({\mathbb R}^n \times {\mathbb R}_+)}$.

\end{itemize}
\end{rem}

%Let $\dot B^\be_{p\si}(\R) = \{ f \in \dot B^\be_p (\R) \, | \, {\rm div}\, f =0 \}$ and $\dot B^\be_{p\si}(\Omega) = \dot B^\be_{p\si}(\R)\big|_{\Omega}$, $\be \in {\mathbb R}$%, and let $\dot B^\be_{p\si 0}(\Omega)= \{ f \in \dot B^\be_{p\si}(\Omega) \, | \, supp \, f \subset \overline{\Omega} \}$
%. Note that for $f \in \dot B^\be_{p\si }(\Omega)$, there exists $\tilde f \in \dot B^\be_{p\si }(\R)$ with $\tilde f|_{\Omega} =f$ such that $\|\tilde f\|_{\dot B^\be_{p\si}(\R)} \leq c\| f\|_{\dot B^\be_{p\si}(\Omega)}$. %In addition, for $f \in \dot B^\be_{p\si 0}(\Omega)$, zero extension $\tilde f$ over $\R$ of $f$ is in $\dot B^\be_{p\si}(\R)$.

%For the initial condition $h$, we use function spaces as follows:
%\begin{align*}
% \dot {\mathcal B}_{p\si}^{\be}(\R_+)= \left\{\begin{array}{l}
%    \dot B_{p\si}^{\be}(\R_+)\quad \be > \frac1p,\\
%          \dot {B}_{p\si 0}^{\be}(\R_+)\quad \be \leq \frac1p.
%           \end{array}
%           \right.
%\end{align*}
%
%
% For the boundary condition $g$, we use function spaces as follows:
% \begin{align*}
% \dot {\mathcal B}^{\be, \frac{\be}2}_p (\Rn \times {\mathbb R}_+) =\left\{\begin{array}{l}
% \dot B^{\be, \frac{\be}2}_p (\Rn \times {\mathbb R}_+)\quad \be > \frac 2p,\\
%  B^{\be, \frac{\be}2}_{p0} (\Rn \times {\mathbb R}_+)\quad \be \leq \frac2p.
%  \end{array}
%  \right.
%  \end{align*}
%Finally, we  define  the function space $ {\mathcal B}^{\alpha,\frac{\alpha}{2}}_{q;0}(\Omega\times (0,T))$ by  the completion of $C^\infty_0(\Omega\times (0,T))$ with respect to the norm of  ${B}^{\alpha,\frac{\alpha}{2}}_{q}(\Omega\times (0,T))$.

Now, we introduce a notion of weak
solutions of the Navier--Stokes equations with nonzero boundary data.

\begin{defin}[Weak solution to the Stokes equations]
\label{stokesdefinition}
Let $0< \al<2$, and let $h,g, f={\rm div}\, {\mathcal F}$ satisfy the same hypothesis as given in Theorem \ref{thm-stokes}.
Then, vector field $u\in \dot B^{\al,\frac{\al}{2}}_q(\R_+\times {\mathbb R}_+)$ is called a weak solution of the Stokes system \eqref{maineq-stokes} if the following conditions are satisfied:\\
 \begin{align*}
-\int^\infty_0\int_{\R_+} u \cdot\big( \Delta \Phi + D_t \Phi \big) dxdt&= - \int^\infty_0\int_{\R_+} {\mathcal F}:\nabla \Phi dxdt -\int_{\R_+}h(y)\Phi(y,0)dy\\
&\quad%\\
-\int^\infty_0\int_{\Rn}g(x',t) \cdot \frac{\partial \Phi}{\partial x_n}(x',t)dx'dt
\end{align*}
for each $\Phi\in C^\infty_c(\overline{\R_+}\times [{\mathbb R}_+))$ with ${\rm div}_x\, \Phi=0$, $\Phi|_{x_n=0}=0$.

%Here  if $\alpha q>2$, then $<h,\Phi(\cdot,0)>=\int_{\R_+}h(y)\Phi(y,0)dy$ , otherwise, $<h,\Phi(\cdot,0)>$  is the duality paring between $\dot{B}^{\alpha-\frac{2}{q}}_{q}(\R_+)$ and $\dot{B}^{-\alpha+\frac{2}{q}}_{q'a}(\R_+)$.
%Likewise, if $\alpha q>1$, then $<g,\frac{\partial \Phi}{\partial x_n}>=\int^\infty_0\int_{\Rn} g(x',t) \cdot \frac{\partial \Phi}{\partial x_n}(x',t)dx' dty$ , otherwise, $<g,\frac{\partial \Phi}{\partial x_n}>$  is the duality paring between $\dot{B}^{\alpha-\frac{1}{q},\frac{\alpha}{2}-\frac{1}{2q}}_{q}(\Rn\times {\mathbb R}_+)$ and $\dot{B}^{-\alpha+\frac{1}{q},-\frac{\alpha}{2}+-\frac{1}{2q}}_{q}(\Rn\times {\mathbb R}_+).$

\end{defin}

\begin{defin}[Weak solution to the Navier--Stokes equations] Let $0<\al<2$.
Let $h,g$ satisfy the same hypothesis as in Theorem \ref{thm3}.
Then, a vector field $u\in \dot B^{\al,\frac{\al}{2}}_q(\R_+\times {\mathbb R}_+)$ is called a weak solution of the Navier--Stokes system \eqref{maineq2} if the following conditions are satisfied:\\
 \begin{align*}
-\int^\infty_0\int_{\R_+} u \cdot\big( \Delta \Phi + D_t \Phi \big) dxdt &=\int^\infty_0\int_{\R_+}  (u\otimes u):\nabla \Phi  dxdt -\int_{\R_+} h(x) \cdot \Phi(x,0) dx\\
&\quad%\\
-\int^\infty_0\int_{\Rn} g(x',t) \cdot \frac{\partial \Phi}{\partial x_n}(x',t)dx' dt
\end{align*}
for each $\Phi\in C^\infty_c(\overline{\R}_+\times [{\mathbb R}_+))$ with ${\rm div}_x\Phi=0$, $\Phi|_{x_n=0}=0$.

%If  $\frac{n+2}{\al+1}< q<\frac{1}{\al}$, then he term
%$\int^\infty_0\int_{\Rn}g(x',t) \cdot\frac{\partial \Phi}{\partial x_n}(x',t)dx'dt$
% should be replaced by
%$<g,\frac{\partial \Phi}{\partial x_n}>$, where $<\cdot,\cdot>$
%is duality pairing between $\dot B^{\al-\frac{1}{q},\frac{\al}{2}-\frac{1}{2q}}_{q0}(\Rn\times {\mathbb R}_+)$ and $\dot B^{-\al+\frac{1}{q}, -\frac{\al}{2}+\frac{1}{2q}}_{q'}(\Rn\times {\mathbb R}_+)$.

\end{defin}

\begin{rem}

If $ 0< \al < \frac1q$, then the term
$\int^\infty_0\int_{\Rn}g(x',t) \cdot \frac{\partial \Phi}{\partial x_n}(x',t)dx'dt$
 should be replaced by
$<g,\frac{\partial \Phi}{\partial x_n}>$, where $<\cdot,\cdot>$ %$<\cdot, \cdot>_{\R_+}$
is the duality pairing between  $\dot B^{\al-\frac{1}{q},\frac{\al}{2}-\frac{1}{2q}}_{q}(\Rn\times {\mathbb R}_+)$ and $\dot B^{-\al+\frac{1}{q}, -\frac{\al}{2}+\frac{1}{2q}}_{q'}(\Rn\times {\mathbb R}_+)$.

Similarly, if $ 0< \al < \frac2q$, then the term
$\int_{\R_+}h(x) \cdot  \Phi (x, 0)dx$
 should be replaced by
$<h,\Phi(\cdot, 0)>$, where $<\cdot,\cdot>$ %$<\cdot, \cdot>_{\R_+}$
is the duality pairing between  $\dot B^{\al-\frac2q}_{q}(\R_+ )$ and $\dot B^{-\al+\frac2q}_{q'}(\R_+ )$.

%Since $\dot B^{\al, \frac{\al}2}_q (\R_+ \times {\mathbb R}_+) \subset L^1_{loc}(\overline{\R_+} \times [0, \infty))$
%
%Note that by Besov inequality,  we have $\dot B^{\al, \frac{\al}2}_q (\R_+ \times {\mathbb R}_+) \subset L^r(\R_+ \times {\mathbb R}_+)$  for   $-\frac{n+2}{r} = \al -\frac{n+2}q$ (see (4) of  Proposition \ref{prop2} and Remark \ref{rem0208}). Hence, for $0 < \al<\frac{n+2}q$.

\end{rem}

\section{\bf Preliminaries.}

\label{preliminary}
\setcounter{equation}{0}

\subsection{Basic Theories}

%
%
%{\color{red}{
%According to the usual trace theorem, if $u\in \dot B^s_p(\R_+)$ ($u\in \dot W^s_p(\R_+)$), then $u|_{x_n=0}\in \dot B_p^{s-\frac{1}{p}}(\Rn)$ for $s>\frac{1}{p}$ (Theorem 6.6.1 in \cite{BL}).
%
%}}

Let $P_{x_n}$ be the Poisson operator defined by
\[P_{x_n}f(x)=
c_n\int_{\Rn}\frac{x_n}{(|x'-y'|^2+x_n^2)^{\frac{n}{2}}}f(y')dy'.\]
Note that $P_{x_n} f$ is a harmonic function in $\R_+$ and $P_{x_n} f|_{x_n =0} = f$ on $\Rn$.

\begin{prop}\label{prop0221-2}
Let $1 < p <\infty$ and $\al  >  0$. If $f \in \dot B^{\al -\frac1p}_p (\Rn)$, then $P_{x_n} f  \in  \dot
B^{\al }_{p}(\R_+)$ and
\begin{equation}\label{CK300-april10}
\| P_{x_n} f \|_{L^{p}(\R_+ )}\leq c \| f \|_{\dot B^{ -\frac1p}_p(\Rn )}, \quad
\| P_{x_n} f \|_{\dot B^{\al}_{p}(\R_+ )}\leq c \| f \|_{\dot B^{\al -\frac1p}_p(\Rn )}.
\end{equation}

\end{prop}
It is well known that $P_{x_n}$ is bounded from  $\dot{B}^{-\frac{1}{p}}_{p}(\Rn)$ to $L^p({\mathbb R}^n_+)$; this was proved in \cite[Lemma 2.1]{KS2}.
Through an interpolation argument, we obtained the second inequality of \eqref{CK300-april10} for $\al>0$.

In addition, for a solenoidal vector field $u\in L^p(\R_+)$,
the following trace theorem holds (See \cite{galdi} for the proof).
\begin{prop}\label{tracetheorem}
Let $0 <\al$ and  $1<p<\infty$. Let
$u \in L^p(\R_+)$ such that ${\rm div} \, u =0$. Then, $u_n \in \dot B^{-\frac1p}_p (\Rn)$ with
\begin{align*}
\| u_n\|_{\dot B^{-\frac1p}_p(\Rn)} \leq c \| u\|_{L^p(\R_+)}, \quad \| u_n\|_{\dot B^{\al -\frac1p}_p(\Rn)} \leq c \| u\|_{\dot B^\al_p (\R_+)}.
\end{align*}

\end{prop}

\begin{proof}
The first inequality is a well-known result (see \cite{galdi} for the proof).

For $\al > \frac1p$, the second inequality is obtained from a usual trace theorem.

Let $0 < \al \leq \frac1p$ and $f \in \dot B^{-\al +\frac1p}_{p'}(\Rn)$, where $p'$ is the H$\ddot{o}$lder conjugate of $p$, that is, $\frac1p + \frac1{p'} =1$.
Let $\dot B^{-\al}_{p'0}(\R_+)$ be a  dual space of $\dot B^\al_p (\R_+)$ and  $<\cdot, \cdot>$ be a duality pairing between $\dot B^\al_p (\R_+)$ and $\dot B^{-\al}_{p'0}(\R_+)$.
Then, from Proposition \ref{prop0221-2}, we have
\begin{align*}
< u_n|_{x_n =0}, f> = \int_{\R_+} u(x) \cdot \na P_{x_n} f (x) dx\\
\leq \| u\|_{\dot B^\al_p (\R_+)} \| \na P_{x_n} f  \|_{\dot B^{-\al}_{p'0}(\R_+)}\\
\leq \| u\|_{\dot B^\al_p (\R_+)} \| P_{x_n} f  \|_{\dot B^{-\al +1}_{p'}(\R_+)}\\
\leq \| u\|_{\dot B^\al_p (\R_+)} \|  f  \|_{\dot B^{-\al +\frac1p}_{p'}(\Rn)}.
\end{align*}
Hence, the proof of the second inequality of Proposition \ref{tracetheorem} is completed.
\end{proof}

It is well known that the Riesz transforms in $\R$, $R_i$, where $1 \leq i \leq n$, are bounded from $\dot B^s_p(\R)$ to $\dot B^s_p(\R)$ for $s\in {\mathbb R}$
\cite{St}. According to the definition of the homogeneous anisotropic Besov space $\dot B^{s,\frac{s}{2}}_{q}({\mathbb R}^n\times {\mathbb R})$ and the multiplier theorem,
the following boundedness property holds true for homogeneous anisotropic Besov spaces.
\begin{prop}
\label{propriesz2}
Let $1<q<\infty$. Then,
\begin{align*}
\|R_i f\|_{ \dot B^{s,\frac{s}{2}}_{q}({\mathbb R}^n\times {\mathbb R})}&\leq c\|f\|_{ \dot B^{s,\frac{s}{2}}_{q}({\mathbb R}^n\times {\mathbb R})},\,\,s\in {\mathbb R}, \,\, 1 \leq i \leq n.
\end{align*}

\end{prop}

We say that a distribution $f$ in $\Rn \times {\mathbb R}_+$ is contained in function space $ \dot B^{\frac{\al}{2}}_q({\mathbb R}_+;\dot{B}^{-\frac{1}{q}}_q(\Rn)), \, 0 < \al < 2$ if $f$ satisfies
$$\| f\|_{\dot B^{\frac{\al}{2}}_q({\mathbb R}_+;\dot{B}^{-\frac{1}{q}}_q(\Rn))} : = \Big( \int_0^\infty \int_0^\infty \frac{ \| f(\cdot, t) -f(\cdot, s)\|^q_{\dot B^{-\frac1p}_q (\Rn)}}{ |t-s|^{1 + \frac{\al}2 q}  }     dsdt \Big)^\frac1q <\infty.$$

\begin{prop}
\label{proppoisson2}
Let $0< \al< 2$ and $1< q < \infty$.
%If
%$f\in L^q({\mathbb R}_+; \dot{B}^{-\frac{1}{q}}_q(\Rn)),$ then,
%$
%P_{x_n}f\in  \dot  L^q(\R_+\times {\mathbb R}_+)
%$ with
%\begin{align*}
%\| P_{x_n} f\|_{L^q(\R_+\times {\mathbb R}_+)} \leq c \|f\|_{L^q({\mathbb R}_+;\dot{B}^{-\frac{1}{q}}_q(\Rn))}.
%\end{align*}
%
%Further, i
If $f\in L^q({\mathbb R}_+; \dot{B}^{\al-\frac{1}{q}}_q(\Rn))\cap \dot B^{\frac{\al}{2}}_q({\mathbb R}_+;\dot{B}^{-\frac{1}{q}}_q(\Rn))$, then,
$
P_{x_n}f\in  \dot  B^{\al,\frac{\al}{2}}_q(\R_+\times {\mathbb R}_+)
$
with
\begin{align*}
\|P_{x_n}f\|_{\dot B^{\al,\frac{\al}{2}}_q(\R_+\times {\mathbb R}_+)}&\leq c \big( \|f\|_{L^q({\mathbb R}_+;\dot{B}^{\al-\frac{1}{q}}_q(\Rn))}
+ \|f\|_{\dot B_q^{\frac{\al }{2}}({\mathbb R}_+;\dot{B}_q^{-\frac{1}{q}}(\Rn))} \big).
\end{align*}
\end{prop}

\begin{proof}
From Proposition \ref{prop0221-2}, we have
\begin{align*}
\| P_{x_n} f\|_{L^q (0, \infty; \dot B^\al_q (\R_+))} \leq c \|f\|_{L^q({\mathbb R}_+;\dot{B}^{\al-\frac{1}{q}}_q(\Rn))}%,\\
%\| P_{x_n} f\|_{L^q(\R_+\times {\mathbb R}_+)} \leq c \|f\|_{L^q({\mathbb R}_+;\dot{B}^{-\frac{1}{q}}_q(\Rn))}
.
\end{align*}
 Proposition \ref{prop0221-2} also gives the estimate \begin{align*}
\|P_{x_n} f(x',t) - P_{x_n} f(x',s)\|_{L^q(\R_+)}
\leq c   \| f(\cdot, t) -f(\cdot, s)\|_{\dot B^{-\frac1q}_q (\Rn)}.
\end{align*}
Hence we have
\begin{align*}
\| P_{x_n} f\|_{L^q (\R_+; \dot B^{\frac{\al}2}_q {\mathbb R}_+)}
 &=  \Big(\int_{\R_+} \int_0^\infty \int_0^\infty \frac{|P_{x_n} f(x',t) - P_{x_n} f(x',s)|^q  }{|t-s|^{1 +\frac{\al}2q}  } dsdt dx  \Big)^\frac1q\\
& \leq c  \Big( \int_0^\infty \int_0^\infty \frac{ \| f(\cdot, t) -f(\cdot, s)\|^q_{\dot B^{-\frac1q}_q (\Rn)}}{ |t-s|^{1 + \frac{\al}2 q}  }     dsdt \Big)^\frac1q  \\
& = c \| f\|_{\dot B^{\frac{\al}{2}}_q({\mathbb R}_+;\dot{B}^{-\frac{1}{q}}_q(\Rn))}.
\end{align*}
Hence, from (2) of Proposition \ref{prop2}, we complete the proof of  Proposition \ref{proppoisson2}.
\end{proof}

\subsection{Estimates of the heat operator}

The fundamental solutions of the heat and Laplace equations in $\R$ are denoted by $  \Ga$ and $N$, respectively, that is,
\[
 \Gamma(x,t)=\left\{\begin{array}{ll} \vspace{2mm}
 \frac{1}{ (2\pi t)^{\frac{n}{2}}}e^{-\frac{|x|^2}{4t}}&\mbox{if }t>0,\\
 0& \mbox{if }t\leq 0,
 \end{array}\right.  \mbox{ and }     N(x) = \left\{\begin{array}{ll}
 \vspace{2mm}
  \frac{1}{\om_n (2-n)|x|^{n-2}}&\mbox{if }n\geq 3,\\
 \frac{1}{2\pi}\ln |x|&\mbox{if }n=2.\end{array}\right.
 \]

We define heat operators $T_1, T_2,T_1^*$ and $T_2^*$, respectively, as follows:
\begin{align*}
 T_1f(x,t)&=\int^t_{-\infty}\int_{\R}\Gamma(x-y,t-s)f(y,s)dyds, \\
T_2g(x,t)&=\int^t_{-\infty}\int_{\Rn}\Gamma(x'-y',x_n,t-s)g(y',s)dy'ds,\\
T_1^*f(y,s) & = \int^{\infty}_s \int_{\R} \Ga(x-y, t-s) f(x,t) dxdt,\\
T_2^*g(y,t) &=\int^\infty_s \int_{\Rn} \Ga(x'-y', y_n, t-\tau)
g(x', \tau) dx' dt.
\end{align*}

The following estimates for $T_1$ and $T_1^*$ are derived in Appendix \ref{appendixa}.
 \begin{lemm}\label{lemma0115}
 Let $1<p<\infty$  and $ 0< \al < 2$. Then,
\begin{align*}
%\| T_1f\|_{L^p (\R \times {\mathbb R})} + \| T^*_1 f\|_{L^p (\R \times {\mathbb R})}
%&\leq c \| f\|_{\dot B^{-2, -1}_p (\R \times {\mathbb R})},\\
\| T_1f\|_{\dot B^{\al, \frac{\al}{2}}_p (\R \times {\mathbb R})} + \| T^*_1 f\|_{\dot B^{\al, \frac{\al}{2}}_p (\R \times {\mathbb R})}
&\leq c \| f\|_{\dot B^{\al-2, \frac{\al}{2}-1}_p (\R \times {\mathbb R})}.
\end{align*}
\end{lemm}

The following estimates for $T_2$ and $T_2^*$ are derived in Appendix \ref{appendixb}.
\begin{lemm}
\label{lem-T}
 Let $1<p<\infty$ and $ 0< \al < 2$.
Then,
\begin{align*}
%\| T_2g \|_{L^q(\R_+ \times {\mathbb R})}+\| T_2^*g \|_{L^q(\R_+ \times {\mathbb R})} & \leq c \| g\|_{\dot{B}^{-1-\frac{1}{q},-\frac{1}{2}-\frac{1}{2q}}_q(\Rn \times {\mathbb R})},\\
\| T_2g \|_{\dot{B}^{\al,\frac{\al}{2}}_q(\R_+ \times {\mathbb R})}+\| T_2^*g \|_{\dot{B}^{\al,\frac{\al}{2}}_q(\R_+ \times {\mathbb R})} & \leq c \| g\|_{\dot{B}^{\al-1-\frac{1}{q},\frac{\al-1}{2}-\frac{1}{2q}}_q(\Rn \times {\mathbb R})}.
\end{align*}
\end{lemm}

The following estimates for $\Gamma_t*h:=\int_{\R}\Gamma(x-y,t)h(y)dy$ and $\Gamma_t*h|_{x_n=0}:=\int_{\R}\Gamma(x'-y',y_n,t)h(y)dy$ are derived in Appendix \ref{appendixc}.
\begin{lemm}
\label{proheat1}
Let $1<q<\infty$  and $ 0< \al < 2$.
Then,
\begin{align*}
%\|\Gamma_t*h\|_{L^q(\R\times {\mathbb R}_+)}\leq
%c \|h\|_{\dot B^{-\frac{2}{q}}_q(\R)},\\
\|\Gamma_t*h\|_{\dot B^{\al,\frac{\al}{2}}_q(\R\times {\mathbb R}_+)}\leq
c \|h\|_{\dot B^{\alpha-\frac{2}{q}}_q(\R)}.
\end{align*}

Moreover, $\Gamma_t*h|_{x_n=0}\in \dot B^{\al-\frac{1}{q},\frac{\al}{2}-\frac{1}{2q}}_q(\Rn\times {\mathbb R}_+)$ with  \begin{align*}
\|\Gamma_t*h|_{x_n=0}\|_{\dot B^{\al-\frac{1}{q},\frac{\al}{2}-\frac{1}{2q}}_q(\Rn\times {\mathbb R}_+)}\leq
c \|h\|_{\dot B^{\alpha-\frac{2}{q}}_q(\R)}.
\end{align*}
\end{lemm}

The following estimates for $D_x\Gamma *f:=\int^t_0\int_{\R}D_x \Ga(x-y,t-s)f(y,s)dyds$ and $D_x\Gamma*f|_{x_n=0}:=\int^t_0\int_{\R}D_x \Ga(x'-y',y_n,t-s)f(y,s)dyds$ are derived in
 Appendix \ref{appendixd}.
%\begin{lemm}
%\label{propheat2}
%Let $1 <p <  q<\infty$,  $0<\al  <2$ and $\al -1 -\frac{n+2}q = \be -\frac{n+2}p$. If $0<\al \leq \frac1q$, we assume $\frac{n+1}p > \frac{n+2}q -\al$. Let   $f\in
%\dot B^{\be,\frac{\be}2}_{p0}(\R \times {\mathbb R}_+)$.  Then $D_x\Ga*f|_{x_n=0}\in B^{\al-\frac{1}{q},\frac{\al}{2}-\frac{1}{2q}}_{q0}(\Rn\times {\mathbb R}_+)$
%with
%\begin{align*}
%\|D_x\Ga*f|_{x_n=0}\|_{ \dot B^{\al-\frac{1}{q},\frac{\al}{2}-\frac{1}{2q}}_q(\Rn\times {\mathbb R}_+)}\leq c \|f\|_{ \dot B^{\be,\frac{\be}2}_{p0}(\R \times {\mathbb R}_+)}.
%\end{align*}
%
%
%%\item[(2)]
%%Let $0<\al \leq \frac1q$ and $\frac{n+1}p > \frac{n+2}q -\al$. Let   $f\in
%%\dot B^{\be,\frac{\be}2}_{p0}(\R \times {\mathbb R}_+)
%%$. Then $D_x\Ga*f|_{x_n=0}\in B^{\al-\frac{1}{q},\frac{\al}{2}-\frac{1}{2q}}_{q0}(\Rn\times {\mathbb R}_+)$
%%with
%%\begin{align*}
%%\|D_x\Ga*f|_{x_n=0}\|_{ \dot B^{\al-\frac{1}{q},\frac{\al}{2}-\frac{1}{2q}}_q(\Rn\times {\mathbb R}_+)}\leq c \|f\|_{ \dot B^{\be,\frac{\be}2}_{p}(\R \times {\mathbb R}_+)}.
%%\end{align*}
%
%
%
%\end{lemm}
%
%
%
\begin{lemm}
\label{propheat2-1}
Let $1< q<\infty$ and $0<\al<2$. Further, let
  $f\in L^p({\mathbb R}_+;\dot{B}^\be_p(\R))
$ for some $(\be, p)$ with $p\leq q$, $0<\be<\al\leq \be+1<2$, and $1-\al+\be-(n+2)(\frac{1}{p}-\frac{1}{q})=0$.
Then, $D_x\Ga*f\in {B}^{\al,\frac{\al}{2}}_{q(0)}(\R\times {\mathbb R}_+)$ with
\begin{align*}
\|D_x\Ga*f\|_{\dot{B}^{\al,\frac{\al}{2}}_{q}(\R\times {\mathbb R}_+)}
&\leq
c \|f\|_{L^{p}({\mathbb R}_+;\dot B^\be_p(\R))}.
\end{align*}

Moreover, if $\al+\frac{n+1}{p}-\frac{n+2}{q}>0$, then $D_x\Ga*f|_{x_n=0}\in \dot B^{\al-\frac{1}{q},\frac{\al}{2}-\frac{1}{2q}}_{q(0)}(\Rn\times {\mathbb R}_+)$
with
\begin{align*}
\|D_x\Ga*f|_{x_n=0}\|_{ \dot B^{\al-\frac{1}{q},\frac{\al}{2}-\frac{1}{2q}}_{q(0)}(\Rn\times ({\mathbb R}_+ ))}
\leq c \|f\|_{L^{p}({\mathbb R}_+;\dot B^\be_{p}(\R))}.
\end{align*}

\end{lemm}

\section{\bf Stokes equations \eqref{maineq-stokes} with $f=0$ and $h=0$  and $g_n=0$}

\label{zero}

\setcounter{equation}{0}

%Let $(w,r)$ be the solution of the equations
% \begin{align}\label{maineq1}
%\begin{array}{l}\vspace{2mm}
%w_t - \De w + \na r =0, \qquad {\rm div } \, w=0, \mbox{ in }
% \R_+\times {\mathbb R}_+,\\
%\hspace{30mm}w|_{t=0}= 0, \qquad  w|_{x_n =0} = G.
%\end{array}
%\end{align}
Let
\begin{align*}
K_{ij}(x,t) & = -2 \delta_{ij}D_{x_n}  \Ga(x,t)  +4 D_{x_j}\int_0^{x_n} \int_{\Rn}  D_{z_n}  \Ga(z,t)  D_{x_i} N(x-z)  dz.
\end{align*}
In \cite{sol1}, an explicit formula was formulated for the solution $w$ of the Stokes equations \eqref{maineq-stokes} with $f =0$, $h =0$, and boundary data $g=(g', 0)$ by

\begin{align}\label{simple}
w_i(x,t)& = \sum_{j=1}^{n-1}\int_0^t \int_{\Rn} K_{ij}( x'-y',x_n,t-s)g_j(y',s) dy'ds.
\end{align}
  \begin{theo}\label{Rn-1}
Let $0< \alpha < 2$ and $1 < q < \infty$.
 In addition, let $w$ be the vector field defined by  \eqref{simple} for $g \in \dot B^{\al-\frac{1}{q},\frac12\al-\frac{1}{2q} }_{q(0)} ({\mathbb R}^{n-1} \times {\mathbb R}_+)
$ with $g_n=0$.
Then, $w\in \dot  {B}^{\al,\frac\al 2 }_{q}({\mathbb R}^n_+\times {\mathbb R}_+)$ with
\begin{align*}
\| w\|_{ \dot {B}^{\al,\frac\al 2 }_{q}({\mathbb R}^n_+\times {\mathbb R}_+)}
 \leq
c \|g\|_{ \dot B^{\al-\frac{1}{q},\frac\al 2-\frac{1}{2q}}_{q(0)}({\mathbb
R}^{n-1} \times {\mathbb R}_+)}.
\end{align*}

\end{theo}

\begin{proof}

From (1) of Remark \ref{rem0101-1}, the zero extension  $\tilde{g}$ of $g$ satisfies
\[ \|\tilde{g}\|_{\dot B^{\al-\frac{1}{q},\frac{\al}{2}-\frac{1}{2q}}_{q}(\Rn\times {\mathbb R})}\leq c\|g\|_{\dot B^{\al-\frac{1}{q},\frac{\al}{2}-\frac{1}{2q}}_{q(0)}(\Rn\times {\mathbb R}_+)}.
\]

$w$ can be rewritten through the following form
\begin{equation}\label{C60K-april7}
w_i = -D_{x_n} T_2 \tilde{G}_i   -4\delta_{in}   \Big(\sum_{j=1}^{n-1} R'_jD_{x_n} T_2\tilde{g}_j\Big)+ 4
    \frac{\partial}{\partial x_i} {\mathcal S}\Big(\sum_{j=1}^{n-1} \frac{\pa}{\partial x_j}D_{x_n} T_2\tilde{g}_j\Big),
\end{equation}
  $ i=1,\cdots,n$, where $R'=(R_1',\cdots,R_{n-1}')$ is the $n-1$ dimensional Riesz operator and
${\mathcal S}$ is defined by
\begin{align}
%%\label{f.zero3} {\mathcal I}f (x') &:= \int_{{\mathbb R}^{n-1}} N(x' -y', 0) f(y')
%%dy',\\
%\label{f.zero4}
{\mathcal S}f (x) &:= \int_0^{x_n} \int_{{\mathbb R}^{n-1}} N(x-y) f (y) dy.
%%%\label{f.zero5}
%%\Gamma_t*f(x,t) &:=   \int_{{\mathbb R}^n}\Ga(x-y,t)f(y)dy,\\
%%\label{f.zero5}
%%\Gamma^*_t*f(x,t) &:=   \int_{{\mathbb R}^n}\Ga(x-y^*,t)f(y)dy.
\end{align}
%\begin{align*}%\label{C70K-april7}
%%{\mathcal T}f(x,t)&=\int_{-\infty}^t \int_{\Rn} D_{x_n}\Ga(x'-y', x_n, t-\tau)
%%f(y', \tau) dy' d\tau,\\
%% {\mathcal I}f (x,t) &= \int_{{\mathbb R}^{n-1}} N(x' -y', 0) f(y',x_n,t)
%%dy',\\
%%\label{C80K-april7}
%{\mathcal S}f (x,t) &= \int_0^{x_n} \int_{{\mathbb R}^{n-1}} N(x-y) f (y,t) dy.
%\end{align*}

%Apply real interpolation theory to \eqref{v-42} and \eqref{v-42-1}, for $0 <\al <2$, we have
%\begin{align}
%\label{y22}
% \notag&\| {\mathcal T}   \tilde{G}\|_{\dot B^{\al, \frac12 \al}_{q(0)}(\R_+ \times (0, T))}  \leq c
%\|  \tilde{G}\|_{\dot B^{ \al-\frac1q,  \frac12 \al -\frac1{2q}}_{q}(\Rn \times {\mathbb R}) }
%\leq c\| {G}\|_{\dot B^{ \al-\frac1q,  \frac12 \al -\frac1{2q}}_{q(0)}(\Rn \times {\mathbb R}_+) }.
%\end{align}

%
%Direct computation shows that  for $1 \leq j \leq n-1$
%\begin{equation}\label{C100K-april7-1}
% {\mathcal I}\Big(\sum_{j=1}^{n-1} \frac{\pa}{\partial x_j}{\mathcal T}\tilde{G}_j\Big) = \sum_{j=1}^{n-1}R^{'}_j {\mathcal T}\tilde{G}_j,
%\end{equation}
%where   $R'=(R_1',\cdots,R_{n-1}')$ is $n-1$ dimensional Riesz operator.
Based on the property of the Riesz operator,
 we have
\begin{align}%\label{3-0-v}
\label{v-31}
 \|   \sum_{j=1}^{n-1}R'_j D_{x_n} T_2 \tilde{g}_j \|_{\dot {B}^{\al,\frac{\al}{2}}_{q} ({\mathbb R}^n_+\times {\mathbb R})}  &  \leq c
\sum_{j=1}^{n-1}  \| D_{x_n} T_2 \tilde{g}_j
\|_{\dot {B}^{\al,\frac{\al}{2}}_{q} ({\mathbb R}^{n}_+\times {\mathbb R})}.%\\
%& \leq  c   \|\tilde G\|_{ \dot B^{ \al-\frac1q, \frac{\al}{2} - \frac1{2q}}_{q}(\Rn \times {\mathbb R} )}.
\end{align}

Let $f=\sum_{j=1}^{n-1} \frac{\pa}{\partial x_j}D_{x_n} T_2\tilde{g}_j$.
Direct computation also shows that
${\mathcal S}f$ solves
\begin{equation}\label{CK100-april7}
\De {\mathcal S}f(t)=
\mbox{div}\,F(t)
      \mbox{ in }\,\,{\mathbb R}^n_+\mbox{ for each }t>0,\quad {\mathcal S}f|_{x_n=0}=0,
\end{equation}
where
\begin{equation}\label{C100K-april7}
F_j: = -\frac12D_{x_n} T_2\tilde{g}_j,\qquad j=1\cdots, n-1,
\quad
F_n: = \sum_{j=1}^{n-1} R'_j D_{x_n} T_2\tilde{g}_j(x,t).
\end{equation}
According to the well-known result for the elliptic partial differential equation \cite{adn, adn1},
 solution ${\mathcal S}f$ of the Laplace equation \eqref{CK100-april7} satisfies the estimate as
\begin{equation}
\label{e1}
\|D_x{\mathcal S}f\|_{ L^q(\R_+ \times {\mathbb R}_+)}\leq \|{\mathcal S}f\|_{L^q (0, \infty; \dot W^{1}_q(\R_+)}\leq c\|F\|_{L^q(\R_+ \times {\mathbb R}_+)}
\end{equation}
and
\begin{equation}
\label{e2}
\|D_x{\mathcal S}f\|_{ L^q (0, \infty; \dot W^{2}_q(\R_+))}\leq \|{\mathcal S}f\|_{L^q (0, \infty; \dot W^{3}_q(\R_+))}\leq c\|F\|_{L^q (0, \infty;\dot W^{2}_q(\R_+))}.
\end{equation}
However, as $D_t{\mathcal S}f$ also satisfies elliptic equation \eqref{CK100-april7} with the right hand side $\mbox{div}D_tF$, we have
\begin{equation}
\label{e3}
\|D_tD_x{\mathcal S}f\|_{ L^q(\R_+ \times {\mathbb R}_+)}\leq \|D_t{\mathcal S}f\|_{L^q (0, \infty; \dot W^{1}_q(\R_+)}\leq c\|D_tF\|_{L^q(\R_+ \times {\mathbb R}_+)}.
\end{equation}
By combining  \eqref{e2} and \eqref{e3}, we obtain
%\begin{equation}
%\label{e4}
%\|D_x{\mathcal S}f\|_{ L^q(\R_+\times {\mathbb R})}\leq  c\|F\|_{L^q(\R_+\times {\mathbb R})}
%\end{equation}
%and
\begin{equation}
\label{e5}
\|D_x{\mathcal S}f\|_{ \dot W^{2,1}_{q}(\R_+\times {\mathbb R})}\leq c\|F\|_{\dot W^{2,1}_{q}(\R_+\times {\mathbb R})}.
\end{equation}
The interpolation of  \eqref{e1} and \eqref{e5} (see (1) of Proposition \eqref{prop2}) gives
\begin{align}
\label{e6}
\|D_x{\mathcal S}f\|_{ \dot B^{\al,\frac{\al}{2}}_{q}(\R_+\times {\mathbb R})}&\leq c\|F\|_{\dot B^{\al,\frac{\al}{2}}_{q}(\R_+\times {\mathbb R})}\leq c\|D_{x_n} T_2\tilde{g}\|_{\dot B^{\al,\frac{\al}{2}}_{q}(\R_+\times {\mathbb R})}
%&\leq  c   \|G\|_{ \dot B^{ \al-\frac1q, \frac{\al}{2} - \frac1{2q}}_{q(0)}(\Rn \times {\mathbb R}_+)}
\end{align}
for $0<\al<2.$
Based on  \eqref{C60K-april7}, \eqref{v-31}, and \eqref{e6}, we conclude that
\begin{equation}\label{interpolation-w}
\| w\|_{\dot  B^{\al, \frac{\al}2 }_{q} (\R_+ \times {\mathbb R}_+)} \leq c \|D_{x_n} T_2 \tilde{g}\|_{\dot B^{\al,\frac{\al}{2}}_{q}(\R_+\times {\mathbb R})} \qquad  0<\al<2.
\end{equation}

%Observe that ${\mathcal T}=D_{x_n}T_2$, where  $T_2$ is the heat operator defined in Section \ref{preliminary}.
From Lemma \ref{lem-T}, we have
\begin{align*}
%\label{v-42}
%\| {\mathcal T}   \tilde{G} \|_{ L^q(\R_+ \times {\mathbb R})}  \leq c\| T_2    \tilde{G} \|_{\dot B^{1,\frac{1}{2}}_{q}(\R_+ \times {\mathbb R})}\leq
%c  \|    \tilde{G} \|_{ \dot B^{-\frac1q, -\frac1{2q}}_{q}(\Rn \times {\mathbb R})},\\
%\label{v-42-1}
\|D_{x_n} T_2    \tilde{g} \|_{\dot B^{\al,\frac{\al}2}_{q}(\R_+ \times {\mathbb R})} % \leq c\| T_2    \tilde{G} \|_{\dot B^{\al+1 ,\frac{\al +1 }{2}}_{q}(\R_+ \times {\mathbb R})}
\leq
c  \|   \tilde{g} \|_{ \dot B^{\al -\frac1q, \frac{\al}2-\frac1{2q}}_{q}(\Rn \times {\mathbb R})}
 \leq
c  \|   g \|_{ \dot B^{\al -\frac1q, \frac{\al}2-\frac1{2q}}_{q(0)}(\Rn \times  {\mathbb R}_+ )}.
\end{align*}
Hence, the proof of Theorem \ref{Rn-1} is completed.
\end{proof}

%\begin{rem}
%\label{regularity-zero}
%Let $G_{ij}^*(x,y,t)= D_{x_j}\int_0^{x_n} \int_{\Rn}   \Ga(z-y^*,t)  D_{x_i} N(x-z)  dz$, $y^*=(y',-y_n)$.
%It is known that \begin{align}
%\label{gauss}
%|D^{s}_{t} D^{k}_{x} D_{y}^{m} G^*_{ij}(x,y,t)|
%& \leq  \frac{c}{t^{s + \frac{m_n}2} (|x-y^*|^2 +t )^{\frac{n+k'+m'}2 } (x_n^2 +t)^{\frac{k_n}2 }}e^{-\frac{cy_n^2}{t}},
%\end{align}
%where $ 1 \leq  i \leq n$ and $1 \leq j \leq n-1$ (See Proposition 2.5 of \cite{Sol-2}).
%Using the properties of
%heat kernel $\Gamma_t$ and the estimates of
%$G^*_{ij}$, we have
%\begin{align*}
%|D^{s}_{t} D^{k}_{x} D_{x_j}K_{ij}(x'-y',x_n,t)| \leq  \frac{c}{t^{s + \frac{1}2} (|x'-y'|^2+x_n^2 +t )^{\frac{n+k'}2 } (x_n^2 +t)^{\frac{k_n}2 }}.
%\end{align*}
%Using this estimate of $K_{ij}$, direct computation shows that
%\[
%\|D_xw(\cdot, x_n,t)\|_{L^q(\Rn)}\leq ct^{\frac{1}{2}}x_n^{-2}\|G\|_{L^q(\Rn \times {\mathbb R}_+)}.
%\]
%and
%\[
%\|D_xw(\cdot, x_n,t)\|_{L^q(\Rn)}\leq ct^{\frac{1}{2}}x_n^{-2+\al-\frac{1}{q}}\|G\|_{\dot{B}^{\al-\frac{1}{q},\frac{\al}{2}-\frac{1}{2q}}_q(\Rn \times {\mathbb R}_+)}.
%\]
%\end{rem}

\section{\bf Proof of Theorem \ref{thm-stokes}  }
\label{general}
\setcounter{equation}{0}

Let us consider the Stokes equations \eqref{maineq-stokes} with nonhomogeneous data $h, f=\mbox{\rm div}{\mathcal F}, g$.
We represent the solution of Stokes equations \eqref{maineq-stokes} according to four vector fields, $v, V, \nabla \phi$, and $w$ as follows.

%\subsection{Existence and estimate for smooth data}
%
%First, we assume that $h\in C^\infty(\R_+)$ and $g\in C^\infty(\Rn\times {\mathbb R}_+)$ with $h|_{x_n=0}=g|_{t=0}$.

\subsection{Solution Representation}
\label{decomposition}

Let $\tilde {\mathcal F}$ be the extension of ${\mathcal F}$ over $\R \times {\mathbb R}_+$ such that $\tilde {\mathcal F} \in L^p({\mathbb R}_+; \dot B^\be_p (\R))$ with $\|\tilde{\mathcal F}\|_{ L^p ( {\mathbb R}_+;  \dot B^{\be}_p(\R  ))}\leq c\|{\mathcal F}\|_{ L^p({\mathbb R}_+; \dot B^{\be}_p(\R_+))} $.
Set $\tilde{f}=\mbox{div }\tilde{\mathcal F}.$
Define $V$ by
 \begin{align}
  V(x,t)=\int^t_{0}\int_{\R}\Gamma(x-y,t-s){\mathbb P}\tilde{f}(y,s)dyds.
 \end{align}
  Here, ${\mathbb P}$ is the Helmholtz projection operator defined on $\R$ defined as
\[
[{\mathbb P}\tilde{f}]_j(x,t)=\delta_{ij}\tilde{f}_i+D_{x_i}D_{x_j}\int_{\R}N(x-y)\tilde{f}_i(y,t)dy=\delta_{ij}\tilde{f}_i+R_iR_j\tilde{f}_i.\]
Observe that
 \[
    \mbox{div}\, {\mathbb P}\tilde{f}=0\mbox{ in }\R\times {\mathbb R}_+\mbox{ and } \tilde{f}={\mathbb P}\tilde{f} +\nabla {\mathbb Q}\tilde{f},\]
  where
\[
{\mathbb Q}\tilde{f}=-D_{x_i}\int_{\R}N(x-y)\tilde{f}_i(y,t)dy.\]
    In addition,
  $V$ satisfies the following equations:
 \begin{align}
 \label{heatequations1}
\begin{array}{l}\vspace{2mm}
V_t - \De V  ={\mathbb P}\tilde{f},\ \mbox{div }V=0 \mbox{ in }
 \R\times {\mathbb R}_+.
 \\
\hspace{30mm}V|_{t=0}= 0\mbox{ on }\R.
\end{array}
\end{align}
Furthermore,
$V$ can be rewritten as
\begin{align}
\label{f1}
  V_j(x,t)=-\int^t_{0}\int_{\R}D_{y_k}\Gamma(x-y,t-s)\Big(\delta_{ij}\tilde{F}_{ki}+R_iR_j\tilde{F}_{ki}\Big)(y,s)dyds.
 \end{align}

Let $\tilde{h}\in \dot {B}^{\al-\frac{2}{q}}_{q\sigma}(\R)$ be the solenoidal extension of $h$ with $\|\tilde{h}\|_{ \dot {B}^{\al-\frac{2}{q}}_q(\R)}\leq c\|{h}\|_{ \dot {\mathcal B}^{\al-\frac{2}{q}}_q(\R_+)}$.
We define $v$ as
  \begin{equation}
  \label{f2}
  v(x,t)=\int_{\R}\Gamma(x-y,t) \tilde{h}(y)dy.
  \end{equation}
Note that $v$ satisfies the following equations:
 \begin{align}
 \label{heatequations2}
\begin{array}{l}\vspace{2mm}
v_t - \De v  =0,\ \mbox{div}\,v=0 \mbox{ in }
 \R\times {\mathbb R}_+,\\
\hspace{30mm}v|_{t=0}= \tilde{h} \mbox{ on }\R.
\end{array}
\end{align}
  Next, we define $\phi$ as
 \begin{equation}
 \label{f3}
 \phi(x,t)=2\int_{\Rn}N(x'-y',x_n)\Big(g_n(y',t)-v_n(x',0,t)-V_n(x',0,t)\Big)dy'.
 \end{equation}

 In addition,
 \begin{equation*}
  \Delta \phi=0, \,\,  \nabla \phi|_{x_n=0}=\Big(R'(g_n-v_n|_{x_n=0}-V_n|_{x_n=0}), g_n-v_n|_{x_n=0}-V_n|_{x_n=0}\Big).
 \end{equation*}
Note that $ \nabla\phi|_{t=0}=0$ if $g_n|_{t=0}=h_n|_{x_n=0}$% if $\al > \frac3q$.

Let $G=(G',0)$, where
\begin{equation}
\label{G}
G'=(G_1,\cdots, G_{n-1})=g'-v'|_{x_n=0}-V'|_{x_n=0}-R'(g_n-v_n|_{x_n=0}-V_n|_{x_n=0}).
\end{equation}
Note that $G'|_{t=0}=0$ if $g|_{t=0}=h|_{x_n=0}$.
 Let $w$ be the vector field defined using  \eqref{simple} with boundary data $G=(G',0)$ for $G'$, as defined in  \eqref{G}.
 Then,
 \begin{equation}\label{representation-u}
 u=w+\nabla \phi+v+V\mbox{ and } p=r-\phi_t+{\mathbb Q}\tilde{f}
 \end{equation}
  formally satisfies the nonstationary Stokes equations \eqref{maineq-stokes}.%

\subsection{%Proof of Theorem \ref{thm-stokes}
Estimates of $u=v+V+\nabla \phi+w$}

\label{proof.stokes}

$\bullet$ %{\color{red}{ Since ${\rm div}\, V=0$ and $ {\rm div}\, v =0$ in $\R_+ \times {\mathbb R}_+$,  by Proposition \ref{tracetheorem},
%\begin{align}
%\label{t1}
%\| V_n|_{x_n =0} \|_{ L^q ({\mathbb R}_+, \dot{B}^{-\frac1q}_q(\Rn))}\leq c \|V\|_{ L^q (\R_+ \times {\mathbb R}_+)},\\
%\label{t1-1}
%  \| v_n|_{x_n =0} \|_{ L^q ({\mathbb R}_+; \dot{B}^{-\frac1q}_q(\Rn))} \leq c \|v\|_{ L^q (\R_+ \times {\mathbb R}_+)}.
%\end{align}
%}}
By applying Proposition \ref{tracetheorem} to $V(t) -V(s)$ and $v(t)-v(s)$, we also have
\begin{align}
\label{t11}\| V_n(t)|_{x_n =0}-V_n(s)|_{x_n =0} \|_{\dot{B}^{-\frac1q}_q(\Rn)}\leq c \|V(t)-V(s)\|_{ L^q (\R_+)},\\
 \label{t111} \| v_n(t)|_{x_n =0}-v_n(s)|_{x_n =0} \|_{ \dot{B}^{-\frac1q}_q(\Rn)} \leq c \|v(t)-v(s)\|_{ L^q (\R_+)}.
\end{align}
From \eqref{t11} and \eqref{t111}, we have
\begin{align}
\label{t41}
\|v_n|_{x_n=0}\|_{\dot {B}_q^{\frac{\al }{2}}({\mathbb R}_+;\dot{B}_q^{-\frac{1}{q}}(\Rn))}\leq \|v\|_{\dot {B}_q^{\frac{\al }{2}}({\mathbb R}_+;L^q(\R_+))}=\|v\|_{L^q(\R_+;\dot {B}_q^{\frac{\al }{2}}{\mathbb R}_+)},\\
\label{t42} \|V_n|_{x_n=0}\|_{ \dot B_q^{\frac{\al }{2}}({\mathbb R}_+;\dot{B}_q^{-\frac{1}{q}}(\Rn))}\leq \|V\|_{\dot {B}_q^{\frac{\al }{2}}({\mathbb R}_+;L^q(\R_+))}=\|V\|_{L^q(\R_+;\dot {B}_q^{\frac{\al }{2}}{\mathbb R}_+)}.
\end{align}

%{\color{red}{
%$\bullet$ From the usual trace theorem
%\begin{align}
%\label{t2}
%\| V_n|_{x_n =0} \|_{ L^q ({\mathbb R}_+; \dot{B}^{2-\frac1q}_q(\Rn))}\leq c \|V\|_{ L^q ({\mathbb R}_+;\dot{W}^2_q (\R_+ ))},\\
% \label{t2-1} \| v_n|_{x_n =0} \|_{L^q ({\mathbb R}_+;  \dot{B}^{2-\frac1q}_q(\Rn))} \leq c \|v\|_{L^q ({\mathbb R}_+; \dot{W}^2_q (\R_+))}.
%\end{align}
%Interpolating \eqref{t1} with \eqref{t2}, and \eqref{t1-1} with \eqref{t2-1}, respectively, we have
%we have
%\begin{align}
%\label{t6}
%\|v_n|_{x_n=0}\|_{L^q({\mathbb R}_+;\dot{B}_q^{\al-\frac{1}{q}}(\Rn))}&\leq c\|v\|_{L^q({\mathbb R}_+;\dot{B}_q^{\al}(\R_+))}, \quad 0<\al<2,\\
%\label{t7}
%\|V_n|_{x_n=0}\|_{L^q({\mathbb R}_+;\dot{B}_q^{\al-\frac{1}{q}}(\Rn))}&\leq c\|V\|_{L^q({\mathbb R}_+;\dot{B}_q^{\al}(\R_+))}, \quad  0<\al<2.
%\end{align}
%}}

$\bullet$
Note that
\begin{align*}
D_{x_n}\phi(x,t)
&=P_{x_n}(g_n-v_n|_{y_n=0}-V_n|_{y_n=0}),\\
D_{x'}\phi(x,t)
&=P_{x_n}R'(g_n-v_n|_{y_n=0}-V_n|_{y_n=0}).
\end{align*}

From Proposition \ref{proppoisson2} and according to the properties of Riesz operator, \eqref{t41} and \eqref{t42}, we have %if $0<\al<1$, then
\begin{align}\label{naphi}
\notag\|\nabla \phi\|_{\dot {B}^{\al,\frac{\al}{2}}_q(\R_+\times {\mathbb R}_+)}&\leq
\|g_n-v_n|_{x_n=0}-V_n|_{x_n=0}\|_{L^q({\mathbb R}_+;\dot{B}^{\al-\frac{1}{q}}_q(\Rn))}\\
\notag &\quad+c\|g_n-v_n|_{x_n=0}-V_n|_{x_n=0}\|_{\dot {B}_q^{\frac{\al }{2}}({\mathbb R}_+;\dot{B}_q^{-\frac{1}{q}}(\Rn))}\\
\notag&\leq c \Big( \|g_n\|_{L^q({\mathbb R}_+;\dot{B}^{\al-\frac{1}{q}}_q(\Rn))}+\|g_n\|_{\dot {B}_q^{\frac{\al }{2}}({\mathbb R}_+;\dot{B}_q^{-\frac{1}{q}}(\Rn))}\\
&\quad+\|v\|_{\dot {B}^{\al,\frac{\al}{2}}_q(\R_+\times {\mathbb R}_+)}+\|V\|_{\dot {B}^{\al,\frac{\al}{2}}_q(\R_+\times {\mathbb R}_+)} \Big).
\end{align}

%
%According to  Lemma \ref{propheat2-1},     $V|_{x_n=0}\in  \dot B^{\al-\frac{1}{q},\frac{\al}{2}-\frac{1}{2q}}_{q(0)}(\Rn\times {\mathbb R}_+)$  with\\
%$
%\|V|_{x_n=0}\|_{ \dot B^{\al-\frac{1}{q},\frac{\al}{2}-\frac{1}{2q}}_{q(0)}(\Rn\times {\mathbb R}_+)}\leq c \|{\mathcal F}\|_{L^p({\mathbb R};\dot B^{\be}_p(\R))}$.

From Lemma \ref{proheat1} and Lemma \ref{propheat2-1},
 we have $v|_{x_n=0} \in  \dot  B^{\al-\frac{1}{q},\frac{\al}{2}-\frac{1}{2q}}_{q}(\Rn\times {\mathbb R}_+)$ and $ \ V|_{x_n =0}\in  \dot  B^{\al-\frac{1}{q},\frac{\al}{2}-\frac{1}{2q}}_{q(0)}(\Rn\times {\mathbb R}_+)$. Together with  \eqref{compatibility1} and \eqref{homo0221}, we conclude that $g-v|_{x_n=0}-V|_{x_n=0}\in \dot B^{\al-\frac{1}{q},\frac{\al}{2}-\frac{1}{2q}}_{q(0)}(\Rn\times {\mathbb R}_+).$
This again implies that
$R'(g-v|_{x_n=0}-V|_{x_n=0})\in \dot B^{\al-\frac{1}{q},\frac{\al}{2}-\frac{1}{2q}}_{q(0)}(\Rn\times {\mathbb R}_+).$
In the end, we conclude that $G'=g'-v'|_{x_n=0}-V|_{x_n=0}-R'(g_n-v_n|_{x_n=0}-V_n|_{x_n=0})\in \dot B^{\al-\frac{1}{q},\frac{\al}{2}-\frac{1}{2q}}_{q(0)}(\Rn\times {\mathbb R}_+)$
with
$$
\|G'\|_{\dot B^{\al-\frac{1}{q},\frac{\al}{2}-\frac{1}{2q}}_{q(0)}(\Rn\times {\mathbb R}_+)}\leq c\Big(\|g\|_{ \dot B^{\al-\frac{1}{q},\frac{\al}{2}-\frac{1}{2q}}_{q}(\Rn\times {\mathbb R}_+)}  +\|v\|_{\dot {B}^{\al,\frac{\al}{2}}_q(\R_+\times {\mathbb R}_+)}
+\|V\|_{\dot B^{\al,\frac{\al}{2}}_{q}(\R_+\times {\mathbb R}_+)} \Big).$$

$\bullet$
By applying Theorem \ref{Rn-1} to the fact that $G=(G',0)\in \dot B_{q(0)}^{\al-\frac{1}{q},\frac{\al}{2}-\frac{1}{2q}}(\Rn\times {\mathbb R}_+)$, we conclude that $w\in \dot {B}_q^{\al,\frac{\al}{2}}(\R_+\times {\mathbb R}_+)$
with
\begin{align}\label{w0221}
\notag\| w\|_{ \dot {B}^{\al,\frac\al 2 }_{q}({\mathbb R}^n_+\times {\mathbb R}_+)}
 &\leq c \|G\|_{  \dot B^{\al-\frac{1}{q},\frac\al 2-\frac{1}{2q}}_{q(0)}({\mathbb
R}^{n-1} \times {\mathbb R}_+)}\\
&\leq c\Big(\|g\|_{ \dot B^{\al-\frac{1}{q},\frac{\al}{2}-\frac{1}{2q}}_{q}(\Rn\times {\mathbb R}_+)}
+\|v\|_{\dot {B}^{\al,\frac{\al}{2}}_q(\R_+\times {\mathbb R}_+)}+\|V\|_{\dot B^{\al,\frac{\al}{2}}_{q}(\R_+\times {\mathbb R}_+)} \Big).
 \end{align}

$\bullet$
From  \eqref{representation-u}, \eqref{naphi}, and \eqref{w0221}, as well as Lemma \ref{proheat1}
and from Lemma \ref{propheat2-1}, the proof of the estimate in Theorem \ref{thm-stokes} for smooth $(h,g)$ with $h|_{x_n=0}=g|_{t=0}$ is completed.

\subsection{Uniqueness}

Let $\tilde{u}\in \dot{B}^{\al,\frac{\al}{2}}_q(\R_+\times {\mathbb R}_+)$ be another solution of the Stokes equations with the same data. Then
\begin{align*}
-\int^\infty_0\int_{\R_+} (u-\tilde{u}) \cdot\big( \Delta \Phi + D_t \Phi \big) dxdt= -0
\end{align*}
for any $\Phi\in C^\infty_c(\overline{\R_+}\times [{\mathbb R}_+))$ with ${\rm div}_x\, \Phi=0$, $\Phi|_{x_n=0}=0$.

Suppose that $\al-\frac{n+2}{q}=-\frac{n+2}{r}$ for some $r$ with $1<r<\infty$, then $\dot{B}^{\al,\frac{\al}{2}}_q(\R_+\times {\mathbb R}_+)\hookrightarrow L^r(\R_+\times {\mathbb R}_+)$.
There is $\Phi \in L^{\frac{r}{r-1}}(\R_+\times {\mathbb R}_+)$ satisfying that $\Delta \Phi + D_t \Phi+\nabla \Pi=|u-\tilde{u}|^{r-2}(u-\tilde{u})$, $\mbox{div}\,\Phi=0$, and $\Phi|_{x_n=0}=0$.
And this implies that $u-\tilde{u}=0$ in $L^r(\R_+\times {\mathbb R}_+)$, and this again implies that $u=\tilde{u}$ almost everywhere in $\R_+\times {\mathbb R}_+$.
Therefore, the solution is unique in the class $\dot{B}^{\al,\frac{\al}{2}}_q(\R_+\times {\mathbb R}_+)$ when $\al-\frac{n+2}{q}=-\frac{n+2}{r}$ for some $r$ with $1<r<\infty$.

\section{Nonlinear problem}

\label{nonlinear}
\setcounter{equation}{0}

In this section, we give the proof of Theorem \ref{thm3}. Accordingly, we construct approximate solutions and then
derive uniform convergence in homogeneous anisotropic Besov spaces $\dot B^{\al,\frac{\al}{2}}_q(\R_+\times {\mathbb R}_+)$. For the uniform estimates,
 bilinear estimates should be preceded.

\subsection{\bf H$\ddot{\rm o}$lder type inequality}

The following H$\ddot{\rm o}$lder type inequality in Besov space is well known
(See Lemma 2.2 in \cite{chae}). %For the H$\ddot{\rm o}$lder type inequality in anisotropic Besov space see  \cite{grubb3,yamazaki}).
\begin{prop}\label{bilinear1}
Let $\be >0$, $\frac1{r_i} + \frac1{s_i} = \frac1p$, and $i=1,2$. Then,
\begin{align*}
\| fg\|_{\dot B^{\be,\frac{\be}2}_{pq} (\R \times {\mathbb R} ) }  \leq
        c \big( \| f\|_{ \dot  B^{\be,\frac{\be}2}_{r_1 q} (\R \times {\mathbb R}) } \| g\|_{  L^{s_1} (\R \times {\mathbb R})}   +c \| f\|_{ L^{r_2 }(\R \times {\mathbb R})} \|g\|_{ \dot  B^{\be,\frac{\be}2}_{s_2 q}(\R\times {\mathbb R}  )}  \big).
\end{align*}
\end{prop}

%For $f \in  \dot  B^{\be,\frac{\be}2}_{r_1 q} (\R_+ \times {\mathbb R}_+) \cap  L^{r_2 }(\R_+ \times   {\mathbb R}_+)$ and $ g \in  \dot  B^{\be,\frac{\be}2}_{s_2 q}(\R_+ \times  {\mathbb R}_+) \cap  L^{s_1} (\R_+ \times {\mathbb R}_+) )$,
Let $f$ and $g$ be functions defined in $\R_+ \times {\mathbb R}_+$.
Further, let $\tilde f$ and $\tilde g$ be the reflective extensions over $ \R \times {\mathbb R}$ with respect to space and time of $f$ and $g$, respectively. Then, by applying Proposition \ref{bilinear1} to $\tilde f$ and $\tilde g$ for $ 0 < \be < 1$, we obtain
\begin{align}\label{exten-sepa}
\| fg\|_{\dot B^{\be,\frac{\be}2}_{pq} (\R_+ \times {\mathbb R}_+ ) }  \leq
        c \big( \| f\|_{ \dot  B^{\be,\frac{\be}2}_{r_1 q} (\R_+ \times {\mathbb R}_+) } \| g\|_{  L^{s_1} (\R_+ \times {\mathbb R}_+)}   +c \| f\|_{ L^{r_2 }(\R_+ \times {\mathbb R}_+)} \|g\|_{ \dot  B^{\be,\frac{\be}2}_{s_2 q}(\R_+ \times{\mathbb R}_+ )}  \big).
\end{align}

\subsection{\bf Proof of Theorem \ref{thm3}}

In this section, we show the construction of a solution of Navier--Stokes equations \eqref{maineq2}.

\subsubsection{Approximate solutions}

Let $(u^1,p^1)$ be the solution of the equations
\begin{align}
\label{maineq5-1}
\begin{array}{l}\vspace{2mm}
u^1_t - \De u^1 + \na p^1 =0, \qquad {\rm div} \, u^1 =0, \mbox{ in }
 \R_+\times {\mathbb R}_+,\\
\hspace{30mm}u^1|_{t=0}= h, \qquad  u^1|_{x_n =0} = g.
\end{array}
\end{align}
Let $m\geq 1$.
After obtaining $(u^1,p^1),\cdots, (u^m,p^m)$, construct $(u^{m+1}, p^{m+1})$, which satisfies the equations
\begin{align}
\label{maineq5}
\begin{array}{l}\vspace{2mm}
u^{m+1}_t - \De u^{m+1} + \na p^{m+1} =f^m, \qquad {\rm div} \, u^{m+1} =0, \mbox{ in }
 \R_+\times {\mathbb R}_+,\\
\hspace{30mm}u^{m+1}|_{t=0}= h, \qquad  u^{m+1}|_{x_n =0} = g,
\end{array}
\end{align}
where $f^m = {\mathcal F}^m=-\mbox{div}(u^m\otimes u^m)$.

\subsubsection{Uniform boundedness}
Let $0<\al<2$ and $q=\frac{n+2}{\al+1}$. Moreover, let $h$ and $g$ satisfy the hypothesis in Theorem \ref{thm3}. Hence, $h, g,$ and ${\mathcal F}^m$ satisfy the hypothesis in Theorem \ref{thm-stokes}.
Set
\begin{align*}
M_0&=\|h\|_{  \dot B^{\al-\frac{2}{q}}_{q}({\mathbb
R}^{n}_+)}+
\|g\|_{ \dot  B^{\al-\frac{1}{q},\frac\al 2-\frac{1}{2q}}_{q}({\mathbb
R}^{n-1} \times {\mathbb R}_+)}\\
&\quad +\|g_n\|_{ \dot  {B}^{\frac12 \al
}_q ({\mathbb R}_+; \dot B^{-\frac1q}_q ({\mathbb
R}^{n-1}))}+\|g_n\|_{ L^q({\mathbb R}_+;\dot{B}^{\al-\frac{1}{q}}_q(\Rn))}.
\end{align*}

Observe that $0<-\al+\frac{n+2}{q}<n+2$, so the solution of \eqref{maineq5-1}  exists uniquely in $\dot{B}^{\al,\frac{\al}{2}}_q(\R_+\times {\mathbb R}_+)$.
 By applying Theorem \ref{thm-stokes} to the solution of  \eqref{maineq5-1}, we have
\begin{align}
\label{uc1}
 \| u^{1}\|_{\dot B^{\al,\frac\al 2 }_{q}({\mathbb R}^n_+\times {\mathbb R}_+)}
\leq c_1 M_0.%  \Big( \|h\|_{ \dot B^{\al-\frac{2}{q}}_{q}({\mathbb
%R}^{n}_+)}
%+ \|g\|_{  \dot B^{\al-\frac{1}{q},\frac\al 2-\frac{1}{2q}}_{q}({\mathbb
%R}^{n-1} \times {\mathbb R}_+)}
% \\
%&\quad +\|g_n\|_{   \dot {B}^{\frac12 \al
%}_q ({\mathbb R}_+; \dot B^{-\frac1q}_q ({\mathbb
%R}^{n-1}))}+\|g_n\|_{ L^q({\mathbb R}_+;\dot{B}^{\al-\frac{1}{q}}_q(\Rn))} \Big).
\end{align}

%Choose  $\tilde{u}^m\in \dot B^{\al,\frac{\al}{2}}_q(\R\times {\mathbb R}_+)$ an extension of $u^m\in \dot B^{\al,\frac{\al}{2}}_q(\R_+\times {\mathbb R}_+)$ with
%$\|\tilde{u}^m\|_{ \dot B^{\al,\frac{\al}{2}}_q(\R\times {\mathbb R}_+)}\leq c\|{u}^m\|_{ \dot B^{\al,\frac{\al}{2}}_q(\R_+\times {\mathbb R}_+)}$, $m\geq 2$.
Take $1 < p<%\infty$ satisfying $%\frac{q(n+2)}{q +n +2} <
%p <
\min (q,\frac{n+2}{2})$ and let $\be=-2+\frac{n+2}{2}$.
Then
\begin{align*}
1<p\leq q,\
1-\al+\be -\frac{n+2}p +\frac{n+2}q=0,\\
  0 <  \be < \al < \be +1 < 2,\quad
 - \al+\frac{n+1}{p}-\frac{n+2}q >0.
 \end{align*}
Hence, $(\be, p)$ satisfies the assumption of Theorem \ref{thm-stokes}.

From Besv  embedding theorem(see (3) of Proposition \ref{prop2}) it holds that
\begin{align*}
\dot B^{\al,\frac{\al}{2}}_q(\R_+\times {\mathbb R}_+)\hookrightarrow L^{n+2}(\R_+\times {\mathbb R}_+),\\
\dot B^{\al,\frac{\al}{2}}_q(\R_+\times {\mathbb R}_+)\hookrightarrow \dot B^{\be,\frac{\be}{2}}_{ \frac{p(n+2)}{n+2-p} p}(\R_+\times {\mathbb R}_+).
\end{align*}
%Let $1 \leq \al <2$ so that $q =\frac{n+2}{1 +\al} < \frac{n+2}2$. Let  $p < q$ and $\be = -2 +\frac{n+2}p >0$.
In Proposition \ref{bilinear1}, %\eqref{exten-sepa},
by considering $ s_1 = r_2 = n+2$ and $r_1 = s_2 = \frac{p(n+2)}{n+2 -p}$ and based on (3) of Proposition \ref{prop2}, we have
\begin{align}\label{1114-1}
\notag\|u^m\otimes u^m\|_{\dot B^{\be,\frac{\be}2}_{p}(\R_+ \times {\mathbb R}_+ )}
&\leq  c \big( \| u^m\|_{ \dot  B^{\be,\frac{\be}2}_{r_1 p } (\R_+ \times  {\mathbb R}_+) } \| u^m\|_{  L^{s_1} (\R_+ \times  {\mathbb R}_+)}\\
\notag& \qquad    + \| u^m\|_{ L^{r_2}(\R_+ \times  {\mathbb R}_+)} \|u^m\|_{ \dot  B^{\be,\frac{\be}2}_{s_2 p}(\R_+  \times  {\mathbb R}_+  )} \big)\\
&\leq c\|u^m\|^2_{\dot B^{\al,\frac{\al}2}_{q}(\R_+  \times  {\mathbb R}_+)}.
\end{align}

As $   \| u^m\otimes  u^m\|_{ L^p({\mathbb R}_+;\dot B^{\be}_p(\R_+))} \leq c  \| u^m\otimes  u^m\|_{ \dot B^{\be,\frac\be 2}_{p}({\mathbb
R}^{n}_+ \times {\mathbb R}_+)} $, according to Theorem \ref{thm-stokes},   there is $u^{m+1}\in \dot {B}^{\al,\frac{\al}{2}}_q(\R_+\times {\mathbb R}_+)$ satisfying that
\begin{align}
\label{6.4}
\notag \| u^{m+1}\|_{\dot B^{\al,\frac\al 2 }_{q}({\mathbb R}^n_+\times {\mathbb R}_+)}
& \leq   c \Big(  M_0
   + \| u^m\otimes  u^m\|_{ \dot B^{\be,\frac\be 2}_{p}({\mathbb
R}^{n}_+ \times {\mathbb R}_+)}\Big) \\%.
%\end{align}
%
%From \eqref{6.4} and \eqref{1114-1} we have
%\begin{align}
%\label{6.4}
% \| u^{m+1}\|_{\dot B^{\al,\frac\al 2 }_{q}({\mathbb R}^n_+\times {\mathbb R}_+)}
& \leq   c_1 \Big( M_0
   + \|u^m\|^2_{\dot B^{\al,\frac{\al}2}_{q}(\R_+  \times  {\mathbb R}_+)}\Big).
\end{align}
%If $\be \leq 0$, then we take $(\be_1, p_1)$ satisfying $p_1 < p$, $\be_1 >0$ and  $\be_1 -\frac{n+2}{p_1} = \be -\frac{n+2}p = -2$. Then, by Proposition \ref{bilinear1} and Lemma \ref{lemma1114}, we have
%\begin{align}\label{1114-2}
%\notag \|\tilde {u}^m\otimes \tilde{u}^m\|_{\dot B^{\be,\frac{\be}2}_{p}(\R\times {\mathbb R})}
%&\leq  c \|\tilde {u}^m\otimes \tilde{u}^m\|_{\dot B^{\be_1,\frac{\be_1}2}_{p_1 p}(\R\times {\mathbb R})}\\
%\notag &\leq  c\big( \| \tilde {u}^m\|_{ \dot  B^{\be_1,\frac{\be_1}2}_{r_1 p} (\R \times {\mathbb R}) } \| \tilde {u}^m\|_{  L^{s_1} (\R \times {\mathbb R})}   +c \| \tilde {u}^m\|_{ L^{r_2}(\R \times {\mathbb R})} \|\tilde {u}^m\|_{ \dot  B^{\be_1,\frac{\be_1}2}_{s_2 p}(\R\times {\mathbb R}  )}  \big)\\
%&\leq c\|\tilde {u}^m\|^2_{\dot B^{\al,\frac{\al}2}_{p}(\R\times {\mathbb R})} \leq c\|u^m\|^2_{\dot B^{\al,\frac{\al}2}_{p}(\R_+ \times {\mathbb R}_+)}.
%\end{align}
Observe that $0<-\al+\frac{n+2}{q}<n+2$, so the solution of \eqref{maineq5}  exists uniquely in $\dot{B}^{\al,\frac{\al}{2}}_q(\R_+\times {\mathbb R}_+)$.

%Take $(\be, p)$ satisfying $\al -1 -\frac{n+2}q = \be -\frac{n+2}p$ with $\frac{n+1}p > \frac{n+2}q -\al$.
%According to   Proposition \ref{bilinear1} and Lemma \ref{lemma1114}, we have
%\begin{align*}
%\|\tilde {u}^m\otimes \tilde{u}^m\|_{\dot B^{\be,\frac{\be}2}_{p}(\R\times {\mathbb R})}\leq  c \big( \| \tilde {u}^m\|_{ \dot  B^{\be,\frac{\be}2}_{r_1} (\R \times {\mathbb R}) } \| \tilde {u}^m\|_{  L^{s_1} (\R \times {\mathbb R})}   +c \| \tilde {u}^m\|_{ L^{r_2}(\R \times {\mathbb R})} \|\tilde {u}^m\|_{ \dot  B^{\be,\frac{\be}2}_{s_2}(\R\times {\mathbb R}  )}  \big)\\
%\leq c
%\|\tilde{u}^m\|_{\dot B_q^{\al,\frac{\al}{2}} (\R\times {\mathbb R}_+)}^2,
%\end{align*}
%and if $\al q<2$, then there is $(\be,p)$  with $p\leq q$, $0<\be<\al\leq \be+1<2$,   $1-\al+\be-\frac{n+2}{p}+\frac{n+2}{q}>0,   \al+\frac{n+1}{p}-\frac{n+2}{q}>0$ so that
%\[
%\|\tilde{u}^m\otimes \tilde{u}^m\|_{\dot B_q^{\be,\frac{\be}{2}} (\R\times {\mathbb R}_+)}
%\leq c
%\|\tilde{u}^m\|_{\dot B_q^{\al,\frac{\al}{2}} (\R\times {\mathbb R}_+)}^2.
%\]
%Hence, from the estimate of Theorem \ref{thm-stokes}, for $  m\geq 2$ it holds that
%\begin{align}
%\label{6.4}
% \| u^{m+1}\|_{\dot B^{\al,\frac\al 2 }_{q}({\mathbb R}^n_+\times {\mathbb R}_+)}
% \leq
%  c_1 \Big( M_0
%   + \|u^m\|_{ \dot B^{\al,\frac\al 2}_{q}({\mathbb
%R}^{n}_+ \times {\mathbb R}_+)}^2 \Big).
%\end{align}
% Here $0=1-\al+\be-\frac{n}{p_1}-\frac{2}{p_2}+\frac{n+2}{q}$ if $\al q\geq 2$ and $0=1-\al+\be-\frac{n+2}{p}+\frac{n+2}{q}$ if $\al q\geq 2$.

%According to \eqref{uc1},
%\[
%\|u^1\|_{\dot B^{\al,\frac{\al}{2}}_q(\R_+\times {\mathbb R}_+)}\leq c_1M_0.\]
Under the condition that $\|u^m\|_{\dot B^{\al,\frac{\al}{2}}_q(\R_+\times {\mathbb R}_+)}\leq M$ and from  \eqref{6.4}, we have
\[
\|u^{m+1}\|_{\dot B^{\al,\frac{\al}{2}}_q(\R_+\times {\mathbb R}_+)}\leq c_1 \big( M_0+
  M^2 \big).
\]
Choose $  c_1 M \leq\frac12 $ and $M_0$ with  $2c_1M_0\leq M$. Then,  based on the mathematical induction argument, we can conclude that
\[
\|u^{m}\|_{\dot B^{\al,\frac{\al}{2}}_q(\R_+\times {\mathbb R}_+)}\leq M\quad \mbox{ for all }\quad m=1,2\cdots.
\]

\subsubsection{Uniform convergence}

Let $U^m=u^{m+1}-u^m$ and $P^m=p^{m+1}-p^m$.
Then, $(U^m, P^m)$ satisfies the equations
\begin{align}
\label{diff}
&\begin{array}{rl}\vspace{2mm}
U^m_t - \De U^m + \na P^m &=-\mbox{div }\big(u^m\otimes U^{m-1}+U^{m-1}\otimes u^{m-1} \big)\mbox{ in }
 \R_+\times {\mathbb R}_+,\\\
 \mbox{div} \, U^{m}&=0\mbox{ in }
 \R_+\times {\mathbb R}_+,\\
  & U^{m}|_{t=0}= 0, \quad  U^{m}|_{x_n =0} =0.
\end{array}
\end{align}

Since  $0<-\al+\frac{n+2}{q}<n+2$, so the solution of \eqref{diff}  exists uniquely in $\dot{B}^{\al,\frac{\al}{2}}_q(\R_+\times {\mathbb R}_+)$.
From Theorem \ref{thm-stokes}, we have
\begin{align*}
&\|U^m\|_{\dot{B}_q^{\al,\frac{\al}2 }({\mathbb R}^n_+\times {\mathbb R}_+)}
\\
&\leq c_2 (\|u^m\|_{\dot {B}_q^{\alpha,\frac{\al}2}(\R_+\times {\mathbb R}_+)}+\|u^{m-1}\|_{\dot {B}^{\alpha,\frac{\al}2}_q(\R_+\times {\mathbb R}_+)})\|U^{m-1}\|_{\dot {B}_q^{\alpha,\frac{\al}2}(\R_+\times {\mathbb R}_+)}\\
&\leq 2c_2M\|U^{m-1}\|_{\dot {B}_q^{\alpha,\frac{\al}2}(\R_+\times {\mathbb R}_+)}.
\end{align*}
%Here $0=1-\al+\be-\frac{n}{p_1}-\frac{2}{p_2}+\frac{n+2}{q}$ if $\al q\geq 2$ and $0=1-\al+\be-\frac{n+2}{p}+\frac{n+2}{q}$ if $\al q\geq 2$.

Choose $M$ so that $c_2M<\frac{1}{4}$,
then, the above-mentioned estimate results in
\begin{equation}
\label{m2}
\|U^m\|_{\dot {B}_q^{\al,\frac{\al}2 }({\mathbb R}^n_+\times {\mathbb R}_+)}\leq \frac{1}{2}\|U^{m-1}\|_{\dot {B}_q^{\alpha,\frac{\al}2}(\R_+\times {\mathbb R}_+)}.
\end{equation}

 \eqref{m2} implies that the infinite series $\sum_{k=1}^\infty U^k$ converges in $\dot {B}_q^{\alpha,\frac{\al}2}(\R_+\times {\mathbb R}_+)$.
Hence, $u^n=u^1+\sum_{k=1}^nU^{k}, m=2,3,\cdots$  converges to $u^1+\sum_{k=1}^\infty U^{k}$ in $ \dot {B}_q^{\alpha,\frac{\al}2}(\R_+\times {\mathbb R}_+)$.
Now, set $u:=u^1+\sum_{k=1}^\infty U^{k}.$

\subsection{Existence}

Let $u$ be the vector field constructed in the previous section.
In this section, we show that $u$ satisfies a weak formulation of the Navier--Stokes equations. Let $\Phi\in C^\infty_{0}(\overline{{\mathbb R}_+}\times [{\mathbb R}_+))$ with $\mbox{div }\Phi=0$ and $\Phi|_{x_n=0}=0$.
Note that
\begin{align*}
\label{w1}
-\int^\infty_0\int_{\R_+}u^{m+1}\cdot  \big( \Delta \Phi + D_t \Phi\big) dxdt&=\int^\infty_0\int_{\R_+}(u^m\otimes u^m): \nabla\Phi dxdt -\int_{\R_+} h(x) \cdot \Phi(x,0) dx\\
&\quad%\\
-\int^\infty_0\int_{\Rn} g(x',t) \cdot \frac{\partial \Phi}{\partial x_n}(x',t)dx' dt.
\end{align*}
As $\al = -1 +\frac{n+2}q$, by using (3) of Proposition \ref{prop2}, we have $\dot B^\al_q(\R_+ \times {\mathbb R}_+ \subset L^{n+2} (\R_+ \times {\mathbb R}_+)$. Now, send $m$ to infinity, then, as $u^m\rightarrow u$ in $\dot {B}^{\alpha,\frac{\alpha}{2}}_q(\R_+\times {\mathbb R}_+)$, we have
\begin{align*}
%\label{w1}
-\int^\infty_0\int_{\R_+}u\cdot \big( \Delta \Phi  +D_t \Phi \big) dxdt&=\int^\infty_0\int_{\R_+} (u\otimes u): \nabla\Phi dxdt-\int_{\R_+} h(x) \cdot \Phi(x,0) dx\\
&\quad%\\
-\int^\infty_0\int_{\Rn} g(x',t) \cdot \frac{\partial \Phi}{\partial x_n}(x',t)dx' dt.
\end{align*}
Therefore, we conclude that $u$ is a weak solution of  \eqref{maineq2}.

\subsection{Uniqueness}

Let $ v\in \dot B^{\al,\frac{\al}{2}}_q(\R_+\times {\mathbb R}_+)$ be another solution of Naiver--Stokes equations \eqref{maineq2} with pressure $q$. Then,
 $u-v$ satisfies the equations
\begin{align*}
%\begin{array}{c}\vspace{2mm}
(u-v)_t - \De (u-v) + \na (p-q)& =-\mbox{div}(u\otimes (u-v)+(u-v)\otimes v)\mbox{ in }
 \R_+\times {\mathbb R}_+, \\
 {\rm div} \, (u-v)& =0,
 %\end{array}  \mbox{ in }
 \mbox{ in }\R_+\times {\mathbb R}_+,\\
 (u-v)|_{t=0}= 0, &\quad (u-v)|_{x_n =0} =0.
\end{align*}

Note that $u, \, u_1 \in L^{n+2}(\R_+ \times {\mathbb R}_+)$.
Applying  the estimate of Theorem \ref{thm-stokes} in \cite{CJ2} to the above Stokes equations,  we have
%????????????
%for $ $ we have
\begin{align*}
\| u-u_1\|_{L^{n+2}({\mathbb R}^n_+ \times (0, \tau))}
 \leq c \|u\otimes (u-u_1)+(u-u_1)\otimes u_1 \|_{L^{\frac{n+2}{2}}({\mathbb R}^n_+ \times (0, \tau))}\\
 \leq c_3 ( \|u\|_{L^{n+2}({\mathbb R}^n_+ \times (0, \tau))}+\|u_1\|_{L^{n+2}({\mathbb R}^n_+ \times (0, \tau))})\| u-u_1\|_{L^{n+2}({\mathbb R}^n_+ \times (0, \tau))}, \ \tau <\infty.
 \end{align*}
 Since  $u,u_1\in L^{n+2}( \R_+ \times {\mathbb R}_+)$, there is $ 0 <\de$ such that if $\tau_1 < \tau_2$ and $\tau_2 -\tau_1 \leq \de$, then
 \[ \|u\|_{L^{n+2} ({\mathbb R}^n_+ \times (\tau_1, \tau_2))}+\|u_1\|_{{L^{n+2} ({\mathbb R}^n_+ \times (\tau_1, \tau_2))}}<\frac{1}{c_3+1}.\]
 Hence, we have
 % Then we have
\[
 \| u-u_1\|_{L^{n+2}({\mathbb R}^n_+ \times (0, \de)))}
 <\| u-u_1\|_{L^{n+2} ({\mathbb R}^n_+ \times (0,\de))}.\]
 This implies that
 $\| u-u_1\|_{L^{n+2} ({\mathbb R}^n_+ \times (0,\de))}=0$, that is, $u\equiv u_1$ in $\R_+\times (0,\de]$.
Observe that  $u-u_1$ satisfies the Stokes equations
 \begin{align*}
%\begin{array}{c}\vspace{2mm}
(u-u_1)_t - \De (u-u_1) + \na (p-p_1)& =-\mbox{div}(u\otimes (u-u_1)+(u-u_1)\otimes u_1)\mbox{ in }
 \R_+\times (\de,\infty), \\
 \mbox{div } (u-u_1)& =0
 %\end{array}  \mbox{ in }
 \mbox{ in }\R_+\times (\de,\infty),\\
 (u-u_1)|_{t=\de}= 0, &\quad (u-u_1)|_{x_n =0} =0.
\end{align*}
%????????????
Again, applying  the estimate of Theorem \ref{thm-stokes}  in \cite{CJ2} to the above Stokes equations,  we have %$u-u_1$ satisfies the inequality
%????????????
%for $ $ we have
\begin{align*}
\| u-u_1\|_{L^{n+2}(\R_+ \times [\de,2\de])}
& \leq c_3 ( \|u\|_{L^{n+2}(\R_+ \times [\de,2\de])}+\|u_1\|_{L^{n+2}(\R_+ \times [\de,2\de])})\| u-u_1\|_{L^{n+2}(\R_+ \times [\de,2\de])}\\
 & < \| u-u_1\|_{L^{n+2}(\R_+ \times [\de,2\de])}.
 \end{align*}
 This implies that
 $\| u-u_1\|_{L^{n+2}({\mathbb R}^n_+ \times (\de, 2\de))}=0$, that is, $u\equiv u_1$ in $\R_+\times [\de,2\de]$.
 %$ M<\frac{1}{2c_1}$ and $M_0<\frac{M}{2c_1}$
% If we take $T_1\leq \frac{1}{4c_3^2(\|u\|_{L^{q}_\al(0,T;L^p({\mathbb R}^n_+))}+\|v\|_{L^{q}_\al(0,T;L^p({\mathbb R}^n_+))}+1)^2}$ together with $T_1\leq 1$, then the above inequality leads to the conclusion that
% \[
% \| u-v\|_{L^{q}_\al(0,T;L^p({\mathbb R}^n_+))}=0\mbox{ that is, }u\equiv v\mbox{ in }\R_+\times (0,T).\]
% By  Corollary \ref{bilinear},  %we have
%\[
%\|u-v\|_{{B}^{\al,\frac{\al}2 }_{\infty}({\mathbb R}^n_+\times (0,T))}
%%\leq c T^\frac12 \|(u\otimes (u-v)+(u-v)\otimes v)\|_{{B}^{\alpha,\frac{\al}2}_\infty(\R_+\times (0,T))}\]
%%\[
%\leq c_2T^\frac12 (\|u\|_{{B}^{\alpha,\frac{\al}2}_\infty(\R_+\times (0,T))}+\|v\|_{{B}^{\alpha,\frac{\al}2}_\infty(\R_+\times (0,T))})\|u-v\|_{{B}^{\alpha,\frac{\al}2}_\infty(\R_+\times (0,T))} .
%\]
After iterating this procedure  finitely many times, we obtain  the conclusion that $u=u_1$ in $\R_+\times {\mathbb R}_+$.
Therefore, we conclude the proof of  the global in time uniqueness.

\appendix
\setcounter{equation}{0}

\section{Proof of Lemma \ref{lemma0115}}
%\setcounter{equation}{0}
%{\color{red}{
\label{appendixa}

In \cite{lady?}, it was determined that
\[\]
\begin{align}
\label{T_11}
\|T_1f\|_{\dot{W}^{2,1}_q(\R \times {\mathbb R})}&\leq c\|f\|_{L^q(\R\times {\mathbb R})}.
\end{align}

Note that $T_1^*$ is the adjoint operator of $T_1$. Hence,  \eqref{T_11} implies that
\begin{align}
\label{T^*2}
\| T_1^*f\|_{L^p(\R \times {\mathbb R})} \leq  c\|f\|_{\dot W^{-2,-1}_p(\R\times {\mathbb R})}.
\end{align}

Further, note that $D^2_yT_1^*f$ and $\ D_s T_1^*f$ comprise $L^p$ Fourier multipliers as the Fourier transform of $T_1^*f$ is $
\widehat{T^*_1 f}(\xi, \eta)=\frac{1}{|\xi|^2-i \eta}
\hat{f}(\xi,\eta)$.
Hence, we have
\begin{align}
\label{T^*1}
\|T_1^*f\|_{\dot W^{2,1}_{p}(\R \times {\mathbb R})}    \leq c \|f\|_{L^p(\R \times {\mathbb R})}, \quad   1<p<\infty.
\end{align}
As $T^*_1$ is the adjoint operator of $T_1$,  \eqref{T^*1} implies that
\begin{align}
\label{T2}
\| T_1f\|_{L^p(\R \times {\mathbb R})} \leq  c\|f\|_{\dot W^{-2,-1}_p(\R\times {\mathbb R})}.
\end{align}
By applying the real interpolation theory to \eqref{T_11} and \eqref{T2}, and   \eqref{T^*2} and \eqref{T^*1}, we obtain estimates of $T_1f$  and $T_1^* f$ in $\dot B^{\al,\frac{\al}{2}}_q(\R\times {\mathbb R})$ for $0<\al<2$.

%\end{proof}

%}}

\section{Proof of Lemma \ref{lem-T}}
\label{appendixb}
%\setcounter{equation}{0}
%\begin{proof}
%The last inequality  is easy to obtained by  Young's Theorem, and we omit its proof and concentrate ourselves to derive the first two inequalities.
%$\bullet$
First, let us derive the estimate of $T_2g$.
From \cite{lady?}, we have the following estimate
\begin{equation}
\label{T_21}\|T_2g\|_{\dot{W}^{2,1}_q(\R_+ \times {\mathbb R})}\leq c\|g\|_{\dot{B}_q^{1-\frac{1}{q},\frac12-\frac{1}{2q}}(\Rn \times {\mathbb R})}.
\end{equation}
%Using H$\ddot{\rm o}$lder's equality, we have
%\begin{align*}
%\| T_2f \|_{L^q(\R_+ \times {\mathbb R})} = \sup_{\phi\in C^\infty_0(\R_+\times {\mathbb R}) } \frac{\int_{-\infty}^\infty \int_{\R_+} T_2f(x,t)\cdot \phi(x,t) dxdt}{\| \phi \|_{L^{q'}(\R_+ \times{\mathbb R})} }.
%%,
%\end{align*}
%where $\frac1p + \frac1q =1$.
 % such that $supp \, \phi \subset \R_+ \times {\mathbb R}$.
Note that the identity
\begin{align}\label{1027-1}
\int_{-\infty}^\infty \int_{\R_+} T_2g(x,t) \phi(x,t) dxdt = <g, T_1^* \tilde \phi|_{y_n=0}>
\end{align}
holds for $\phi\in C^\infty_0(\R_+\times {\mathbb R})$,
%where $ T_1^*\phi(y,s) = \int_s^\infty \int_{\R}  \Ga(x-y, t-s) \phi(x,t) dx dt$, and
  where $ T^*_1 \tilde \phi$ is defined in Section \ref{preliminary} with zero extension $\tilde \phi$ of $\phi$ and $<\cdot,\cdot>$ is the duality pairing between $\dot{B}^{-1-\frac{1}{q},-\frac{1}{2}-\frac{1}{2q}}_q(\Rn\times {\mathbb R})$ and $\dot{B}^{1+\frac{1}{q},\frac{1}{2}+\frac{1}{2q}}_{q'}(\Rn\times {\mathbb R}).$ %. Note that
%$\int_s^\infty \int_{\R_+}  \Ga(x' -y', x_n, t-s) \phi(x,t) dx dt=\int_s^\infty \int_{\R}  \Ga(x' -y', x_n, t-s) \phi(x,t) dx dt$, since $\phi\in C^\infty_0(\R_+\times {\mathbb R}_+)$.
%Observe that $T_1^*\phi|_{y_n =0} = {\mathcal J}^*\phi(y',s)$, where $T^*$ is the same operator defined in the statement of Proposition \ref{lemma0115}, that is,
%%Define
%$T^*\phi(y,s) = \int_s^\infty \int_{\R}  \Ga(x' -y', x_n-y_n, t-s) \phi(x,t) dxdt.
%$
%?????????????????\\
%
%?????????????
%Since $\phi \in L^{q'}(\R_+ \times {\mathbb R}_+) $,  we get
%From  \eqref{T^*1},
%% In Appendix \ref{appen.point1} we show
%we have
%\begin{align}
%\label{point1}
%\|T_1^*\phi\|_{\dot W^{2,1}_{q'}(\R \times {\mathbb R})}    \leq c \|\phi\|_{L^{q'}(\R \times {\mathbb R}) }=\|\phi\|_{L^{q'}(\R_+ \times {\mathbb R}) }.
%\end{align}
%%?????????????
%Since $1 - \frac1q = \frac1p$, b
Based on (4) of Proposition \ref{prop2} and  \eqref{T^*1}, we have
%we get
\begin{align}\label{1027-2}
\|T_1^*\phi|_{y_n=0}\|_{\dot B^{1+\frac1q, \frac{1}{2}+\frac1{2q}  }_{q'}(\Rn \times {\mathbb R})}
\leq c  \|T_1^*\phi\|_{\dot W^{2,1}_{q'}(\R \times {\mathbb R})}
\leq c \|\phi\|_{L^{q'}(\R_+ \times {\mathbb R}) }.
\end{align}
%Note that $  \Big(  \dot B^{1+\frac1q, \frac{1}{2}+\frac1{2q}  }_{q'}(\Rn \times {\mathbb R})\Big)' = \dot B^{-1-\frac1q, -\frac{1}{2}-\frac1{2q}  }_q(\Rn \times {\mathbb R})$.
By applying the estimates in  \eqref{1027-1} to  \eqref{1027-2}, we have
%from \eqref{1027-1} and \eqref{1027-2},
%we have
%\begin{align*}
%\int_{-\infty}^\infty \int_{\R_+} T_2f(x,t) \phi(x,t) dxdt
%\leq \|f\|_{\dot B^{-1-\frac1q, -\frac{1}{2}-\frac1{2q}  }_q(\Rn \times {\mathbb R}) } \| {\mathcal J}^*\phi\|_{\dot B^{1+\frac1q, \frac{1}{2}+\frac1{2q}  }_{q'}(\Rn \times {\mathbb R}) }\\
%\leq c\|f\|_{\dot B^{-1-\frac1q, -\frac{1}{2}-\frac1{2q}  }_q(\Rn \times {\mathbb R}) } \|\phi\|_{L^{q'}(\R_+ \times {\mathbb R}) }.
%\end{align*}
%This leads to the  estimate
\begin{equation}
\label{mathcalT1}
\| T_2g \|_{L^q(\R_+ \times {\mathbb R})}\leq c\|g\|_{\dot B^{-1-\frac1q, -\frac{1}{2}-\frac1{2q}  }_q(\Rn \times {\mathbb R}) }.
\end{equation}
Further, by applying the real interpolation theory to  \eqref{T_21} and \eqref{mathcalT1},
 %between $\dot{W}^{2k,k}_p(\R\times {\mathbb R})$ and $L^p(\R\times {\mathbb R})$,
 we obtain the estimate of $T_2g$ in $\dot B^{\al,\frac{\al}{2}}_q(\R_+\times {\mathbb R})$ for $0<\al<2$.
%Applying the real  interpolation theorem between $\dot{W}^{2,1}_q(\R_+ \times {\mathbb R})$ and $L^q(\R_+ \times {\mathbb R})$, and between $\dot B^{1-\frac1q, \frac{1}{2}-\frac1{2q}}_q(\Rn \times {\mathbb R})$ and $\dot B^{-1-\frac1q, -\frac{1}{2}-\frac1{2q}}_q(\Rn \times {\mathbb R})$, we complete the proof of  estimate of $T_2f$ for any $0<s<2$.%Lemma \ref{lem-T} .
%Iterating the same argument to the higher derivatives, the estimate could be extended to any $s>0$.

%???????????????
%$\bullet$
Analogously, we can derive the estimate of $T_2^*g$ %Using H$\ddot{\rm o}$lder's equality, we have
%\begin{align*}
%\| T_2^*f \|_{L^q(\R_+ \times {\mathbb R})} = \sup_{\phi\in C^\infty_0(\R_\times {\mathbb R}) }\frac{ \int_{-\infty}^\infty \int_{\R} T_2^*f(y,s)\cdot \phi(y,s) dyds}{\| \phi \|_{L^{q'}(\R_+ \times{\mathbb R})}}.
%%,
%\end{align*}
%where $\frac1p + \frac1q =1$.
by observing that the identity
\begin{align}\label{1027-20}
\int_{-\infty}^\infty \int_{\R_+} T_2^*g(y,s) \phi(y,s) dyds = <g,  T_1 \tilde \phi|_{x_n=0}>
\end{align}
holds for $\phi\in C^\infty_0(\R_+ \times {\mathbb R})$,
where $ T_1 \tilde \phi$ is defined in Section \ref{preliminary} with zero extension $\tilde \phi$ of $\phi$, and $<\cdot,\cdot>$ is the duality pairing between $\dot{B}^{-1-\frac{1}{q},-\frac{1}{2}-\frac{1}{2q}}_q(\Rn\times {\mathbb R})$ and $\dot{B}^{1+\frac{1}{q},\frac{1}{2}+\frac{1}{2q}}_{q'}(\Rn\times {\mathbb R}).$
%Note that $ \int_{-\infty}^t \int_{\R}  \Ga(x' -y', y_n, t-s) \phi(y,s) dy ds= \int_{-\infty}^t \int_{\R_+}  \Ga(x' -y', y_n, t-s) \phi(y,s) dy ds$, since $\phi\in C^\infty_0(\R_\times {\mathbb R}).$
%Observe that $T\phi|_{x_n =0} = {\mathcal J}\phi(x',t)$, where $T$ is the same operator defined in the statement of Proposition \ref{lemma0115}%, that is,
%%Define
%$T\phi(x,t) = \int^t_{-\infty}
% \int_{\R}  \Ga(x' -y', x_n-y_n, t-s) \phi(x,t) dxdt.
%$
By using the same procedure as that used for the estimate of $T_2g$, we can obtain the estimate of $T_2^*g$ as
\begin{equation}
\label{mathcalT1-1}
\| T_2^*g \|_{L^q(\R_+ \times {\mathbb R})}\leq c\|g\|_{\dot B^{-1-\frac1q, -\frac{1}{2}-\frac1{2q}  }_q(\Rn \times {\mathbb R}) }
\end{equation}
(As the procedure is the same as that for $T_2g$, we omitted the details).
As $D_sT_2^*g=T_2^*(D_sg)$, we have
\begin{equation}
\label{mathcalT1-2}
\|D_sT_2^*g \|_{L^q(\R_+ \times {\mathbb R})}\leq c\|D_sg\|_{\dot B^{-1-\frac1q,-\frac12 -\frac1{2q}  }_q(\Rn \times {\mathbb R}) }\leq c\|g\|_{\dot B^{1-\frac1q, \frac{1}{2}-\frac1{2q}  }_q(\Rn \times {\mathbb R}) }.
\end{equation}
%Since $\mbox{div}T_2^*g=0$, $D_{x_n}=
In addition, as $\Delta_y T_2^*g=-D_sT_2^*g$ and $\frac{\partial}{\partial y_n}T_2^*g|_{y_n=0}=g$,
based on the well-known elliptic theory \cite{adn,adn1}, we have
\[
\|T_2^*g(s) \|_{ \dot W^{2}_q(\R_+ )}\leq c\| D_s T_2^*g(s) \|_{L^q(\R_+ )}+c\|g(s)\|_{\dot B^{1-\frac{1}{q}}_q(\Rn)}.
\]
This implies that
\begin{equation}
\label{mathcalT1-3}
\|T_2^*g \|_{L^q({\mathbb R};\dot W^{2}_q(\R_+ ))}\leq  c\|g\|_{ \dot B^{1-\frac1q, \frac{1}{2}-\frac1{2q}  }_q(\Rn \times {\mathbb R}) }.
\end{equation}
By combining  \eqref{mathcalT1-2} and \eqref{mathcalT1-3}, we have
\begin{equation}
\label{mathcalT1-4}
\| T_2^*g \|_{\dot W^{2,1}_q(\R_+ \times {\mathbb R})}\leq c\|g\|_{\dot B^{1-\frac1q, \frac{1}{2}-\frac1{2q}  }_q(\Rn \times {\mathbb R}) }.
\end{equation}
By applying the real interpolation theory to \eqref{mathcalT1-1} and \eqref{mathcalT1-4}, we obtain the estimate of $T_2^*g$ in $\dot B^{\al,\frac{\al}{2}}_q(\R_+\times {\mathbb R})$ for $0<\al<2$. Thus, we complete the proof of Lemma \ref{lem-T}.

%{\color{red}{
%
%Observe that
%\begin{align*}
%D_t T_2f = T_2D_t f,\,\, D^2_{x'} T_2f =T_2D^2_{x'}f,\\
% D_s T_2^*f = T_2^*D_s f,\,\, D^2_{x'} T_2^*f =T_2^*D^2_{x'}f
%\end{align*}
%and
%\begin{align*}
%D_t T_2 f - \De T_2 f =0, \qquad D_s T_2^* f + \De T_2^* f =0.
%\end{align*}
%Hence, by the same argument, we obtain the estimate
%\begin{align}
%\| T_2f\|_{\dot W^{2k,k}_q(\R_+ \times {\mathbb R})}, \quad
%\| T_2^*f\|_{\dot W^{2k,k}_q(\R_+ \times {\mathbb R})}
%\leq c  \|  f\|_{\dot B^{2k-1-\frac1q, \frac{2k -1}{2}-\frac1{2q}}_q(\Rn \times {\mathbb R})}.
%\end{align}
%Applying the real   interpolation theorem and complex interpolation theorem between $\dot{W}^{2k,k}_p(\R_+ \times {\mathbb R})$ and $L^q(\R_+ \times {\mathbb R})$, and between $\dot B^{2k-1-\frac1q, \frac{2k-1}{2}-\frac1{2q}}_q(\Rn \times {\mathbb R})$ and $\dot B^{-1-\frac1q, -\frac{1}{2}-\frac1{2q}}_q(\Rn \times {\mathbb R})$, we complete the proof of    Lemma \ref{lem-T}.
%
%}}

\section{Proof of Lemma \ref{proheat1} }
\label{appendixc}
From \cite{lady?},
the following estimate is known:
\begin{equation}
\label{T_01}
\|\Gamma_t*h\|_{\dot{W}^{2,1}_q(\R \times  {\mathbb R}_+)}\leq c\|h\|_{\dot{B}_q^{2-\frac{2}{q}}(\R)}.
\end{equation}

Let us consider the case where $h\in \dot B^{-\frac{2}{q}}_q(\R)$.
%Let.
%
%Using H$\ddot{\rm o}$lder's equality, we have
%Observe that
%\[
%\|u\|_{L^q(\R\times {\mathbb R}_+)}=\sup_{\phi\in C^\infty(\R\times {\mathbb R}_+)}\frac{\int^T_{0}\int_{\R}u(x,t)\phi(x,t)dx dt}{
%\|\phi\|_{L^{q'}(\R\times {\mathbb R}_+)}}.\]
Note that the identity
$
\int^\infty_{0}\int_{\R}\Gamma_t*h(x,t)\phi(x,t)dx dt=<h,T_1^*\phi|_{s=0}>$ holds for  $\phi\in C_0^\infty(\R\times {\mathbb R})$,%\mbox{ for }\psi\in C^\infty_0(\R\times {\mathbb R}_+),\]
where $T_1^*\phi(y,s)=\int^\infty_{s}\int_{\R}\Gamma(x-y,t-s)\phi(x,t)dxdt,$, and $<\cdot,\cdot>$ is the duality pairing between $\dot{B}^{-\frac{2}{q}}_q(\R)$ and $\dot{B}^{\frac{2}{q}}_{q'}(\R).$
%Observe that $T^*\phi|_{s=0}=u^*$, where $T^*$ is the operator defined in the statement of Proposition \ref{lemma0115}.%, that is, %Define
%$T^*\psi(y,s)=\int^\infty_{s}\int_{\R}\Gamma(x-y,t-s)\psi(x,t)dxdt.$
%Then .
From  \eqref{T^*1}, we have
\[
\|T_1^*\phi\|_{\dot W^{2,1}_{q'}(\R \times {\mathbb R})}
\leq c \|\phi\|_{L^{q'}(\R \times {\mathbb R}) }.\]
%{\color{blue}{
%????????????????
%Take Fourier transform in terms of $x,t$, then
%$
%\hat{T\phi}=\frac{-1}{|\xi|^2+i \eta}
%\hat{\phi}(\xi,\eta)$. Hence
%\[
%\widehat{D_{x_k}D_{x_l}T\phi}=\frac{-\xi_l\xi_k}{|\xi|^2+i \eta}
%\hat{\phi}(\xi,\eta)\]
%and
%\[
%\widehat{D_tT\phi}=\frac{-i \eta}{|\xi|^2+i \eta}
%\hat{\phi}(\xi,\eta).\]
%
%Observe that the multiplier $m_1=\frac{-\xi_n\xi_k}{|\xi|^2+i \eta}, m_2=\frac{-i\eta}{|\xi|^2+i \eta}$ are $L^q$-multipliers in $\R\times {\mathbb R}$ for $1<q<\infty$.
%Therefore by the well known  singular integral operator theory we have
%\begin{align*}
%\|T\phi\|_{\dot W^{2,1}_{q'}(\R \times (0, {\mathbb R}))}    \leq c \|\phi\|_{L^{q'}(\R \times (0, {\mathbb R})) }.
%\end{align*}
%%\end{prop}
%
%
%}}
By using (5) of Proposition \ref{prop2}, this  implies that
\[
\|T_1^*\phi|_{s=0}\|_{\dot{B}_{q'}^{2-\frac{2}{q'}}(\R)}\leq c\|T^*_1\phi\|_{\dot W^{2,1}_{q'}(\R \times {\mathbb R})}
\leq c \|\phi\|_{L^{q'}(\R \times {\mathbb R}_+) }.\]
(See \cite{Tr} and \cite{Triebel2}.)
%Note that $\Big(\dot{B}_{q'}^{2-\frac{2}{q'}}(\R)\Big)'=\dot{B}_q^{-\frac{2}{q}}(\R).$
Hence, we have \[
<h,T_1^*\phi|_{s=0}>\leq c\|h\|_{\dot{B}_q^{-\frac{2}{q}}(\R)}\|T_1^*\phi\|_{\dot{B}^{2-\frac{2}{q'}}(\R)}\leq c\|h\|_{\dot{B}_q^{-\frac{2}{q}}(\R)}\|\phi\|_{L^{q'}(\R \times {\mathbb R}) } .\]
Again, this leads to the following conclusion %\eqref{it1}.
%
%\begin{equation}
%\label{it1}
%\|u\|_{L^q(\R\times {\mathbb R}_+}\leq c\|h\|_{\dot{B}^{-\frac{2}{q}}_q(\R)}. \end{equation}

%Again this implies that
\begin{equation}
\label{it1}
\|\Gamma_t*h\|_{L^q(\R\times  {\mathbb R}_+)}\leq c\|h\|_{\dot{B}^{-\frac{2}{q}}_q(\R)}.
\end{equation}

By interpolating  \eqref{T_01} and \eqref{it1}, we have
\begin{align}
\label{it3}
 \|\Gamma_t*h\|_{B^{\al,\frac{\al}{2}}_q(\R\times {\mathbb R}_+)}\leq c \|h\|_{\dot B^{\al-\frac{2}{q}}_q(\R)}, \qquad 0< \al< 2.
\end{align}
Now, we will derive the estimate of $\Gamma_t*h|_{x_n=0}$.

1) Let $\al>\frac{1}{q}$. Then by (5) of Proposition \ref{prop2},
$\Gamma_t*h\in \dot B^{\al,\frac{\al}{2}}_q(\R\times {\mathbb R}_+)$ implies that
$\Gamma_t*h|_{x_n=0}\in \dot B^{\al-\frac{1}{q},\frac{\al}{2}-\frac{1}{2q}}_q(\Rn\times {\mathbb R}_+)$ with
\begin{align*}
%\label{y3}
\|\Gamma_t*h|_{x_n=0}\|_{\dot B^{\al-\frac{1}{q},\frac{\al}{2}-\frac{1}{2q}}_q(\Rn\times {\mathbb R}_+)}&\leq c\|\Gamma_t*h\|_{\dot B^{\al,\frac{\al}{2}}_q(\R\times {\mathbb R}_+)} \leq c  \| h\|_{\dot B_q^{\al -\frac2 q}(\R)}.
\end{align*}

2) Let  $0 <\al <  \frac{1}{q}.$  In this case, usual trace theorem  does not hold any more.

For $h\in \dot{B}^{\al-\frac{2}{q}}_q(\R)$ the following
 identity holds:
\begin{equation}
\label{holder1}
<\Gamma_t*h\Big|_{x_n=0},\phi>=<h, T_2^*\phi|_{s=0}>,\end{equation} holds for
any  $\phi \in C^\infty_0(\Rn\times {\mathbb R})$,
where $T_2^*\phi(y,s)=\int^\infty_s\int_{\Rn}\Gamma(x'-y',y_n,t-s)\phi(x',t)dx'dt$ and $<\cdot,\cdot>$ is the duality pairing between $\dot{B}^{\al-\frac{2}{q}}_q(\R)$ and $\dot{B}^{-\al+\frac{2}{q}}_{q'}(\R)$ .
%Note that $\int^T_0\int_{\Rn}\Gamma(x'-y',y_n,t)\phi(x',t)dx'dt=\int^\infty_0\int_{\Rn}\Gamma(x'-y',y_n,t)\phi(x',t)dx'dt$ and $\|\phi\|_{\dot{B}^{-\al+\frac{1}{q},-\frac{1}{2}+\frac{1}{2q}}_{q'}(\Rn\times {\mathbb R}_+)}=\|\phi\|_{\dot{B}^{-\al+\frac{1}{q},-\frac{1}{2}+\frac{1}{2q}}_{q'}(\Rn\times {\mathbb R})}$, since $\phi\in C^\infty_0(\Rn\times {\mathbb R}_+)$.
%Observe that
%$u^*(y)=T_2^*\phi(y,s)\Big|_{s=0},$ where $T_2^*$ is the operator defined in the statement of Lemma \ref{lem-T}.%, that is,
%\[
%T_2^*\phi(y,s)=\int^\infty_s\int_{\Rn}\Gamma(x'-y',y_n,t)\phi(x',t)dx'dt.\]
From the result of Lemma \ref{lem-T},
$
T_2^*\phi\in \dot{B}^{-\al+2,-\frac{\al}{2}+1}_{q'}(\R_+\times {\mathbb R}) $ % if $0<\al<\frac{1}{q}$( in $\dot{W}^{2,1}_{q'}(\R_+\times {\mathbb R})$ if $\al=0$)
with
\[
%\left\{\begin{array}{r}
\|T_2^*\phi\|_{ \dot{B}^{-\al+2,-\frac{\al}{2}+1}_{q'}(\R_+\times {\mathbb R})}%\mbox { if }0<\al<\frac{1}{q}
%\\
%\|T_2^*\phi\|_{ \dot{W}^{2,1}_{q'}(\R_+\times {\mathbb R})}\mbox{ if }\al=0
%\end{array}\right\}
\leq c\|\phi\|_{\dot{B}^{-\al+\frac{1}{q},-\frac{\al}{2}+\frac{1}{2q}}_{q'}(\Rn\times {\mathbb R})}%=c\|\phi\|_{\dot{B}^{-\al+\frac{1}{q},-\frac{1}{2}+\frac{1}{2q}}_{q'}(\Rn\times {\mathbb R}_+)}
.\]
By (4) of Proposition \ref{prop2}, this implies that
$T_2^*\phi\Big|_{s=0}\in \dot{B}^{-\al+\frac{2}{q}}_{q'}(\R_+ )$ with
\[
\|T_2^*\phi\Big|_{s=0}\|_{ \dot{B}^{-\al+\frac{2}{q}}_{q'}(\R_+ )}
   %\leq c\|T_2^*\phi\|_{ \dot{B}^{-\al+2,-\frac{\al}{2}+1}_{q'}(\R\times {\mathbb R})}
   \leq c\|\phi\|_{\dot{B}^{-\al+\frac{1}{q},-\frac{\al}{2}+\frac{1}{2q}}_{q'}(\Rn\times {\mathbb R})}.\]
%This implies that
%$u^*\in \dot{B}^{-\al+\frac{2}{q}}_{q'}(\R_+ )$ with
%\[
%\|u^*\|_{ \dot{B}^{-\al+\frac{2}{q}}_{q'}(\Rn\times {\mathbb R}_+ )}\leq c\|\phi\|_{\dot{B}^{-\al+\frac{1}{q},-\frac{1}{2}+\frac{1}{2q}}_{q'}(\Rn\times {\mathbb R}_+)}.\]
Hence
\[
|<h, T_2^*\phi\Big|_{s=0}>|\leq c\|h\|_{\dot{B}^{\al-\frac{2}{q}}_q(\R)}\|\phi\|_{\dot{B}^{-\al+\frac{1}{q},-\frac{\al}{2}+\frac{1}{2q}}_{q'}(\Rn\times {\mathbb R})}\]
%Note that $\Big(\dot{B}^{-\al+\frac{1}{q},-\frac{1}{2}+\frac{1}{2q}}_{q'}(\Rn\times {\mathbb R}_+)\Big)'=\dot{B}^{\al-\frac{1}{q},\frac{1}{2}-\frac{1}{2q}}_{q}(\Rn\times {\mathbb R}_+)$.
Applying the above estimate to \eqref{holder1},
$\Gamma_t*h|_{x_n=0}\in  \dot{B}^{\al-\frac{1}{q},\frac{\al}{2}-\frac{1}{2q}}_{q}(\Rn\times {\mathbb R})$ with
%\[
%\|\Gamma_t*h|_{x_n=0}\|_{\dot{B}^{\al-\frac{1}{q},\frac{1}{2}-\frac{1}{2q}}_{q}(\Rn\times {\mathbb R})}\leq c\|h\|_{\dot{B}^{\al-\frac{2}{q}}_q(\R)}.\]
%This again implies that
%%$\Gamma_t*h|_{x_n=0}\in  \dot{B}^{\al-\frac{1}{q},\frac{1}{2}-\frac{1}{2q}}_{q}(\Rn\times {\mathbb R}_+)$ with
\begin{align}
\label{y1}
\|\Gamma_t*h|_{x_n=0}\|_{\dot{B}^{\al-\frac{1}{q},\frac{\al}{2}-\frac{1}{2q}}_{q}(\Rn\times {\mathbb R})}\leq c\|h\|_{\dot{B}^{\al-\frac{2}{q}}_q(\R)}.\end{align}
%%
%On the other hand, by  Young's inequality we have
%\begin{align}
%\label{y2}
%\| \Gamma_t*h|_{x_n=0}\|_{L^q(\Rn \times (0, T))} \leq c T^{\frac1{2q}} \| h\|_{L^q(\R)}.
%\end{align}
%Recall the fact that $B^{s,\frac{s}{2}}_q(\Omega\times {\mathbb R}_+)=\dot B^{s,\frac{s}{2}}_q(\Omega\times {\mathbb R}_+)+L^q(\Omega\times {\mathbb R}_+)$ for $s<0$. Combining \eqref{y1} and \eqref{y2}, we have
%\begin{align}
%\label{y33}
%\| \Gamma_t*h|_{x_n=0}\|_{B^{ \al-\frac1q,  \frac{\al}{2}-\frac1{2q}}_{q} (\Rn \times (0, T))} \leq c \max\{1,T^{\frac1{2q}}\} \| h\|_{B_q^{ \al-\frac2 q}(\R)}, 0<\al<\frac{1}{q}.
%\end{align}

3) Finally let us consider the case $\al=\frac1q$. Using the real interpolation, we get the case of $\al =\frac1q$.

\section{Proof of Lemma \ref{propheat2-1} }
\label{appendixd}
\setcounter{equation}{0}

$\bullet$
Let $\tilde{f}\in L^p({\mathbb R};{B}^{\be}_p(\R))$ be the zero extension of $f$ to $\R\times {\mathbb R}$.
%Note that
%$D_x\Gamma *f=D_xT_1\tilde{f}.
Note that $D_x \Gamma * \tilde f = \Gamma * D_x \tilde f$.
%Then, according to Lemma \ref{lemma0115}, the following estimates hold
From \eqref{T_11}, we have
\begin{align*}
\|D_x\Gamma * \tilde f\|_{\dot W^{2,1}_p(\R\times {\mathbb R})} \leq c\| D_x \tilde f\|_{L^p( {\mathbb R}; L^p(\R))} \leq c\| f\|_{L^p( {\mathbb R}_+;\dot{W}^1_p(\R))}
\end{align*}
and
\begin{align*}
\|D_x\Gamma *\tilde f\|_{\dot W^{1,\frac{1}{2}}_p(\R\times {\mathbb R})}\leq c \|\Gamma *\tilde f\|_{\dot W^{2,1}_p(\R\times {\mathbb R})} \leq c\| f\|_{L^p( {\mathbb R}_+;L^p(\R))}.
\end{align*}
By interpolating these two estimates, we can obtain
\begin{equation}
\label{itra1}
\|D_x\Ga*{f}\|_{\dot B^{\be+1,\frac{\be +1}{2}}_p(\R\times {\mathbb R})}\leq c\|{ f}\|_{L^p({\mathbb R}_+;\dot{B}^{\be}_p(\R))}, \,\,  0<\be<1.
\end{equation}
Further, by applying Besov imbedding (see (3) of Proposition \ref{prop2}), for $1-\al+\be-(n+2)(\frac{1}{p}-\frac{1}{q})=0$
, we have
\begin{equation}
\label{itr1111}
\|D_x\Ga*{f}\|_{\dot{ B}_{q}^{\al,\frac{\al}{2}}(\R\times {\mathbb R})}\leq c \|{ f}\|_{L^p({\mathbb R}_+;\dot{B}^\be_p(\R))}.
\end{equation}

Note that $D_x\Ga*{f}(x,t)=0$ for $t\leq 0$. Hence, $D_x\Ga*{f}\in \dot { B}_{q(0)}^{\al,\frac{\al}{2}}(\R\times {\mathbb R}_+).$

$\bullet$
Now, we derive the estimate of $D_x\Ga*{f}\Big|_{x_n=0}$.

1)
Let $\al>\frac{1}{q}$. Then, according to the usual trace theorem,
$D_x\Ga*{f}\in B^{\al,\frac{\al}{2}}_{q(0)}(\R\times {\mathbb R}_+)$ implies that
$D_x\Ga*{f}|_{x_n=0}\in B^{\al-\frac{1}{q},\frac{\al}{2}-\frac{1}{2q}}_{q(0)}(\Rn\times {\mathbb R}_+)$ with
\begin{align}
\label{y3}
\notag \|D_x\Ga*{f}|_{x_n=0}\|_{B^{\al-\frac{1}{q},\frac{\al}{2}-\frac{1}{2q}}_{q(0)}(\Rn\times {\mathbb R}_+)} &\leq c\|D_x\Ga*{f}\|_{B^{\al,\frac{\al}{2}}_{q}(\R\times {\mathbb R}_+)}\\
& \leq  c \|{ f}\|_{L^p({\mathbb R}_+ ;\dot{B}^\be_p(\R)))}.%, \al>\frac{1}{q}.
\end{align}

2) Let $0<\al\leq \frac{1}{q}$. In this case, the usual trace theorem does not hold true.

If $\al+\frac{n+1}{p}-\frac{n+2}{q}>0$, we can choose $r$ with $p<r<q,$ $\al+\frac{n+1}{r}-\frac{n+2}{q}>0$. Set $\gamma=\al+\frac{n+2}{r}-\frac{n+2}{q}$, then
$\al-\frac{1}{q}-\frac{n+1}{q}=\gamma-\frac{1}{r}-\frac{n+1}{r}$ and $\al-\frac{1}{q}<\gamma-\frac{1}{r}$. Hence, by using the Besov embedding theorem,
\[
\|D_x\Ga*{f}|_{x_n=0}\|_{B^{\al-\frac{1}{q},\frac{\al}{2}-\frac{1}{2q}}_{q(0)}(\Rn \times {\mathbb R}_+)}\leq c\|D_x\Ga*{f}|_{x_n=0}\|_{B^{\gamma-\frac{1}{r},\frac{\gamma}{2}-\frac{1}{2r}}_{r0}(\Rn \times {\mathbb R}_+)}.\]
As $\gamma>\frac{1}{r}$, the use of the usual trace theorem gives
\begin{align*}
\|D_x\Ga*{f}|_{x_n=0}\|_{B^{\gamma-\frac{1}{r},\frac{\gamma}{2}-\frac{1}{2r}}_{r0}(\Rn \times {\mathbb R}_+)}\leq c\|D_x\Ga*{f}\|_{B^{\gamma,\frac{\gamma}{2}}_{r0}(\Rn \times ({\mathbb R}_+ ))}
\leq c \|{ f}\|_{L^p({\mathbb R}_+;\dot{B}^\be_p(\R))}.
\end{align*}
Hence, the proof of Lemma \ref{propheat2-1} is completed.


\begin{thebibliography}{10}

\bibitem{AF} R. A. Adams and J. J. F. Fournier, {\it Sobolev spaces. Second Edition}, Academic Press (2003).

\bibitem{adn}S. Agmon, A. Douglis \mbox{ and }L. Nirenberg, {\it Estimates near the boundary for solutions of elliptic partial differential equations satisfying general boundary conditions. I}, Comm. Pure Appl. Math.  12, 623-727(1959).


 \bibitem{adn1}S. Agmon, A. Douglis \mbox{ and }L. Nirenberg, {\it Estimates near the boundary for solutions of elliptic partial differential equations satisfying general boundary conditions. II}, Comm. Pure Appl. Math.  17, 35-92(1964).

\bibitem{fernandes} M.F. de Almeida and L.C.F. Ferreira, {\it On the Navier-Stokes equations in the half-space with initial and boundary rough data in Morrey spaces}, J. Differential Equations  254,  no. 3, 1548-1570(2013).

\bibitem{amman-anisotropic} H. Amann, {\it
Anisotropic function spaces and maximal regularity for parabolic problems. Part 1.
Function spaces,}  Jindřich Nečas Center for Mathematical Modeling Lecture Notes, 6. Matfyzpress, Prague, vi+141(2009).

\bibitem{amann}H. Amann, {\it On the strong solvability of the Navier-Stokes equations}, J. Math. Fluid Mech.  2,  no. 1, 16-98 (2000).

\bibitem{amann1} H. Amann, {\it Navier-Stokes equations with nonhomogeneous Dirichlet data}, J. Nonlinear Math. Phys.  10,  suppl. 1, 1-11(2003).

\bibitem{amann2} H. Amann, {\it Nonhomogeneous Navier-Stokes Equations with Integrable Low-Rregularity data,}  Nonlinear problems in mathematical physics and related topics, II,  1–28, Int. Math. Ser. (N. Y.), 2, Kluwer/Plenum, New York, 2002.


\bibitem{BL}  J. Bergh and J. L$\ddot{\rm o}$fstr$\ddot{\rm o}$m, { Interpolation Spaces. An Introduction}, Grundlehren der Mathematischen Wissenschaften, No. 223. Springer-Verlag, Berlin-New York, 1976.

%Interpolation spaces. An introduction. Grundlehren der Mathematischen Wissenschaften, No. 223. Springer-Verlag, Berlin-New York, 1976.

%\bibitem{Bo} M. Bownik,{\it Atomic and molecular decompositions of anisotropic  Besov spaces}, Math. Z.  250,  no. 3, 539-571 (2005).

%\bibitem{CJ} T. Chang and B.J. Jin, {\it Boundary Value Problem of the Nonstationary Stokes System in the Half Space}, Potential Anal.  41,  no. 3, 737-760(2014).

\bibitem{cannone} M. Cannone, F. Planchon and M.Schonbek, {\it Strong solutions to the incompressible Navier-Stokes equations in the half-space}, Comm. Partial Differential Equations  25,  no. 5-6, 903-924(2000).

\bibitem{chae} D. Chae, {\it Local existence and blow-up criterion
for the Euler equations in the Besov spaces}, Asymptotic Analysis, 38,339-358 (2004).

\bibitem{CJ1} T.  Chang and B. Jin, {\it Initial and boundary value problem of the unsteady Navier-Stokes system in the half-space with H$\ddot{\rm o}$lder continuous boundary data}, J. Math. Anal. Appl., 433,  no. 2, 1846-1869  (2016).

\bibitem{chang-jin} T. Chang and B. Jin, {\it Solvability of the Initial-Boundary value problem of the    Navier-Stokes equations  with rough data},  Nonlinear Anal., 125,  498-517  (2015).

\bibitem{CJ2} T. Chang and B. Jin, {\it Initial and boundary values for $L^p_\al(L^p)$  solution of the Navier-Stokes equations in the half-space}, J. Math. Anal. Appl., 439,  no. 1, 70-90   (2016).



%\bibitem{knightly}  J.R. Cannon and G.H. Knightly, {\it A note on the Cauchy problem for the Navier-Stokes equations}, SIAM J. Appl. Math. 18, 641-644(1970).




%\bibitem{BL}  J. Bergh and J. Lofstrom, { \it Interpolation Spaces, An Introduction}, Springer-Verlag, Berlin (1976).



%\bibitem{Bo} M. Bownik, {\it Atomic and molecular
%                   decompositions of anisotropic  Besov spaces},
%                   Math. Z, {\bf 250}, (2005), 539-571.


%\bibitem{C} T.K. Chang, {\it Extension and Restriction theorems
%in anisotropic  Besov spaces },  Commun. Contemp. Math, {\bf 12},
%no. 2, 265-294(2010).


%\bibitem{CC}
%T. Chang and H.J. Choe, {\it  Maximum modulus estimate for the solution of the Stokes equations,}
%J. Differential Equations  254,  no. 7, 2682-2704(2013).


%\bibitem{fabes} E.B. Fabes, B.F. Jones and N.M. Rivi$\grave{\rm e}$re, {\it The initial value problem for the Navier-Stokes equations with data in $L^p$,} Arch.Ration.Mech.Anal.45,222-240(1972).



%\bibitem{CJ} T. Chang and B.J. Jin,{\it Boundary Value Problem of the Nonstationary Stokes System in the Half Space}, Potential Anal.  41,  no. 3, 737-760(2014).


%\bibitem{D} B, Dahlberg, {\it Weighted norm inequalities for the Lusin area integral and the nontangential
%               maximal functions for functions harmonic in a Lipschitz domain},  Studia Math, {\bf 67}, no. 3, 297-314(1980).

%\bibitem{DJK} B. Dahlberg, D, Jerison and C. Kenig, {\it Area integral estimates for elliptic differential operators with
%nonsmooth coefficients}, Ark. Mat,  {\bf 22}, no. 1, 97-108(1984).

%
%\bibitem{DKPV} B. Dahlberg, C. Kenig,  J. Pipher and G. Verchota, {\it Area integral estimates for higher order elliptic equations and equationss},
%                     Ann. Inst. Fourier (Grenoble), {\bf  47}, no. 5, 1425-1461(1997).











%\bibitem{maremonti2}
%F. Crispo and P. Maremonti,{\it On the $(x,t)$  asymptotic properties of solutions of the Navier-Stokes equations in the half-space},  Zap. Nauchn. Sem. POMI.  318, 147-202(2004);translation in  J. Math. Sci. (N. Y.)  136,  no. 2, 3735-3767(2006).


\bibitem{DT} H. Dappa and  H. Triebel, {\it On anisotropic Besov and
                  Bessel Potential spaces}, Approximation and function spaces, 69-87 (Warsaw, 1986), Banach Center
                   Publ., 22, PWN, Warsaw(1989).






%\bibitem{giga} Y. Giga, {\it Solutions for semilinear parabolic equations in $L^p$ and regularity of weak solutions of the Navier-Stokes system,}
%                        J. Diff. Equ. {\bf 61},186-212(1986).



%\bibitem{FJ} M. Frazier and  B. Jawerth, {\it Decomposition of Besov spaces},  Indiana Univ. Math. J., {\bf 34},  no. 4, 777-799(1985).


%\bibitem{farwig1} R. Farwig, R, G.P. Galdi and H. Sohr,{\it Very weak solutions and large uniqueness classes of stationary Navier-Stokes equations in bounded domains of $ R^2$},   J. Differential Equations  227,  no. 2, 564-580(2006).

%\bibitem{farwig1} R. Farwig, H. Sohr and W. Varnhorn, {\it Besov space regularity conditions for weak solutions of the Navier-Stokes equations,} J. Math. Fluid Mech.  16,  no. 2, 307-320(2014).
%Farwig, Reinhard; Kozono, Hideo; Sohr, Hermann Very weak solutions of the Navier-Stokes equations in exterior domains with nonhomogeneous data. J. Math. Soc. Japan  59  (2007),  no. 1, 127–150

\bibitem{farwig4} R. Farwig and H. Kozono, {\it Weak solutions of the Navier-Stokes equations with non-zero boundary values in an exterior domain satisfying the strong energy inequality}, J. Differential Equations 256, no. 7, 2633-2658(2014).

 %\bibitem{farwig5} R. Farwig,G.P. Galdi and H. Sohr,{\it  Very weak solutions and large uniqueness classes of stationary Navier-Stokes equations in bounded domains of $R^2$,} J. Differential Equations  227,  no. 2, 564-580(2006).
%


  \bibitem{farwig6} R. Farwig,G.P. Galdi and H. Sohr, {\it Very Weak Solutions of Stationary and Instationary Navier-Stokes Equations with Nonhomogeneous Data},  Nonlinear elliptic and parabolic problems,  113-136, Progr. Nonlinear Differential Equations Appl., 64, Birkh$\ddot{\rm a}$user, Basel, 2005.


\bibitem{farwig2} R. Farwig, H. Kozono, and H. Sohr, {\it Very weak solutions of the Navier-Stokes equations in exterior domains with nonhomogeneous data}, J. Math. Soc. Japan  59,  no. 1, 127-150(2007).

\bibitem{farwig3} R. Farwig, H. Kozono, and H. Sohr, {\it Global weak solutions of the Navier-Stokes equations with nonhomogeneous boundary data and divergence}, Rend. Semin. Mat. Univ. Padova  125, 51-70(2011).

%\bibitem{FJ} M. Frazier and  B. Jawerth,{\it  Decomposition of Besov spaces}, Indiana Univ. Math. J.  34,  no. 4, 777-799(1985).
\bibitem{galdi} G.P. Galdi, An Introduction to the MAthematical Theory of the Navier-Stokes Equations, vol. I, linearlised steady problems, Springer Tracts in Natural Philosopy vol. 38, Springer 1994.

\bibitem{giga2} Y. Giga, {\it Solutions for semilinear parabolic equations in $L^p$ and regularity of weak solutions of the Navier-Stokes system}, J.Differential Equations 62, no. 2,186-212(1986).


    \bibitem{giga3} Y. Giga and T. Miyakawa, {\it Solutions in $L_r$  of the Navier-Stokes initial value problem},  Arch. Rational Mech. Anal., 89,  no. 3, 267–281 (1985).

%\bibitem{Galdi}
\bibitem{giga}
M. Giga,Y. Giga and H. Sohr, {\it $L^p$   estimates for the Stokes system,}  Functional analysis and related topics, 1991 (Kyoto),  55–67, Lecture Notes in Math., 1540, Springer, Berlin(1993).

 \bibitem{giga1}
 Y. Giga and H. Sohr, {\it Abstract $L^p$   estimates for the Cauchy problem with applications to the Navier-Stokes equations in exterior domains}, J. Funct. Anal. 102, no. 1, 72-94(1991).


%\bibitem{giga} Y. Giga, K. Inui and S. Matsui,{\it On the Cauchy problem for the Navier-Stokes equations with nondecaying initial data,}  Advances in fluid dynamics,  27-68, Quad. Mat., 4, Dept. Math., Seconda Univ. Napoli, Caserta, 1999.
%
%\bibitem{shimizu} Y. Giga, S. Matsui, and Y. Shimizu, {\it On estimates in Hardy spaces for the Stokes flow in a half space}, Math. Z.  231),  no. 2, 383-396(1999).





%\bibitem{HN} S. Hofmann and K. Nystrom, {\it Dirichlet problems for a
%nonstationary linearized system of Navier-Stokes equations in
%non-cylindrical domains} Methods Appl. Anal. ${\bf 9}$, no. 1,
%13-98(2002).
%
%
%
%
%\bibitem{JK} D. Jerison and C. Kenig, {\it The
%                              inhomogeneous Dirichlet Problem in Lipschitz
%                              domains}, J. of Funct. Anal,
%                              $\bf{130}$, 161-219(1995).


\bibitem{grubb1} G. Grubb, {\it Nonhomogeneous time-dependent Navier-Stokes problems in $L_p$   Sobolev spaces,} Differential Integral Equations 8, no. 5, 1013-1046(1995).

  \bibitem{grubb2} G. Grubb, {\it Nonhomogeneous Navier-Stokes problems in $L_p$   Sobolev spaces over exterior and interior domains,}  Theory of the Navier-Stokes equations, 46-63, Ser. Adv. Math. Appl. Sci., 47, World Sci. Publ., River Edge, NJ, 1998.

\bibitem{grubb3} G. Grubb, {\it Nonhomogeneous Dirichlet Navier-Stokes problems in low regularity $L_p$   Sobolev spaces,} J. Math. Fluid Mech. 3, no. 1, 57-81(2001).


\bibitem{grubb} G. Grubb and V.A. Solonnikov, {\it  Boundary value problems for the nonstationary Navier-Stokes equations treated by pseudo-differential methods}, Math. Scand. 69, no. 2, 217-290 (1992).


%\bibitem{HN} S. Hofmann and K. Nystrom, {\it Dirichlet problems for a
%nonstationary linearized system of Navier-Stokes equations in
%non-cylindrical domains} Methods Appl. Anal. ${\bf 9}$, no. 1,
%13-98(2002).

%\bibitem{J} B.  Jones. Jr, {\it Lipschitz Spaces and the heat equation},
%                            J. Math. Mech, {\bf 18}, 379-409 (1968).


%\bibitem{iftimie} D. Iftimie, {\it The resolution of the Navier-Stokes equations in anisotropic spaces}, Revista Mate$\acute{\rm a}$tica Iberoamericana Vol.15, No. 1, 1-35(1999).
%
%
%%\bibitem{JK} D. Jerison and C. Kenig, {\it The
%%                              inhomogeneous Dirichlet Problem in Lipschitz
%%                              domains}, J. of Funct. Anal,
%%                              $\bf{130}$, 161-219(1995).
%
%\bibitem{kato1} T. Kato, {\it Strong solutions of the Navier-Stokes equation in Morrey spaces,} Bol.Soc.Brasil.Mat.(N.S.) 22, no. 2, 127-155(1992)
%
%\bibitem{kato2} T. Kato, {\it Strong $L^p$ solutions of the Navier-Stokes equation in $\R$ with, applications to weak solutions,} Math.Z.187, no.4, 471-480(1984).
%
%
%
%\bibitem{kato} T. Kato and G. Ponce, {\it  Well-posedness of the Euler and Navier-Stokes equations in the Lebesgue spaces $L^p_s (R_2 )$,}
%  Rev. Mat. Iberoamericana  2,  no. 1-2, 73-88(1986).
%
%\bibitem{koch} H. Koch and D. Tataru, {\it Well-posedness for the Navier-Stokes equations}, Adv. Math.  157,  no. 1, 22-35(2001).



\bibitem{KS1} H. Koch and V. A.  Solonnikov,
%{\it Estimates for  solutions of  nonstationary Navier-Stokes problem in anisotropic Soboliv spaces and on
%estimates for the resolvent of the Stokes operator},
{\it $L_p$-Estimates for a solution to the nonstationary Stokes
equations,} Journal of Mathematical Sciences, Vol. 106, No.3,
3042-3072(2001).

\bibitem{KS2} H.  Koch and V. A.  Solonnikov, {\it $L_q$-estimates of the first-order derivatives of solutions to the nonstationary Stokes
problem, Nonlinear problems in mathematical physics and related topics, I}, Int. Math. Ser. (N. Y.), 1, Kluwer/Plenum, New York, 203-218(2002).





\bibitem{kozono1} H. Kozono, {\it Global $L^n$-solution and its decay property for the Navier-Stokes equations in half-space $R^n_+ $,} J. Differential Equations 79, no. 1, 79-88(1989).


%\bibitem{kozono2}H. Kozono, and M. Yamazaki, {\it Semilinear heat equations and the Navier-Stokes equation with distributions in new function spaces as initial data}, Comm. Partial Differential Equations  19,  no. 5-6, 959-1014(1994).

%\bibitem{kozono2} H.Kozono,{\it Strong solution for the Navier-Stokes flow in the half-space}, In The Navier-Stokes equations (Oberwolfach,1988),pages 84-86

%\bibitem{kozono} H. Kozono, T. Ogawa, and Y. Taniuchi,{\it Navier-Stokes equations in the Besov space near $L^\infty$   and BMO,}. Kyushu J. Math.  57,  no. 2, 303-324(2003).

\bibitem{lady?}
O.A. Lady$\check{\rm z}$enskaja, V.A. Solonnikov and N.N. Ura${\rm l}ʹ$ceva, {\it Linear and Quasilinear Equations of Parabolic Type}, (Russian) Translated from the Russian by S. Smith. Translations of Mathematical Monographs, Vol. 23 American Mathematical Society, Providence, R.I. 1968.

%\bibitem{lemarie} P.G. Lemarié-Rieusset, {\it The Navier-Stokes equations in the critical Morrey-Campanato space}, Rev. Mat. Iberoam.  23,  no. 3, 897-930(2007).

\bibitem{lewis} J.E. Lewis, {\it The initial-boundary value problem for the Navier-Stokes equations with data in $L^p$,} Indiana Univ. Math. J. 22, 739-761(1972/73).

%\bibitem{maremonti}
%P. Maremonti,{\it Stokes and Navier-Stokes problems in the half-space: existence and uniqueness of solutions non converging to a limit at infinity,}  Zap. Nauchn. Sem. POMI.  39, 176-240, 362(2008);  translation in  J. Math. Sci. (N. Y.)  159,  no. 4, 486-523(2009).



%
%\bibitem{maremonti3}
%P. Maremonti and G. Staria, {\it On the nonstationary Stokes equations in half-space with continuous initial data},  Zap. Nauchn. Sem. POMI. 33, 118-167,  295(2003); translation in  J. Math. Sci. (N. Y.)  127,  no. 2, 1886-1914(2005).

%\bibitem{maremonti2}
%P.Maremonti and G.Staria, {\it On the $(x,t)$ asymptotic properties of solutions of the Navier-Stokes equations in the half space,} Zap.Nauchn.Sem.POMI,318(2004);translated in J.Math.Sci.136,no.2,3735-3767(3006).


\bibitem{raymond1} J.-P. Raymond, {\it  Stokes and Navier-Stokes equations with nonhomogeneous boundary conditions}, Ann. Inst. H. Poincaré Anal. Non Linéaire 24, no. 6, 921-951(2007).

\bibitem{R} M. Ri, P. Zhang and Z. Zhang, {\it Global well-posedness for Navier-Stokes equations with small initial value in $B^0_{n, \infty} (\Om)$}, J. Math. Fluid Mech., 18, no. 1, 103-131 (2016).


%\bibitem{raymond2} J.-P,Raymond,{\it Feedback boundary stabilization of the three-dimensional incompressible Navier-Stokes equations}, J. Math. Pures Appl. (9)  87,  no. 6, 627-669(2007).


\bibitem{sawada} O. Sawada, {\it On time-local solvability of the Navier-Stokes equations in Besov spaces}, Adv. Differential Equations 8, no. 4, 385-412(2003).

%\bibitem{Sh}  Z. Shen, {\it Boundary value problems for parabolic Lame
%systems and a nonstationary linearized system of Navier-Stokes
%equations in Lipschitz cylinders}, Amer. J. Math. $\bf{ 113}$, no.
%2, 293-373(1991).





%\bibitem{serrin}J. Serrin, {\it On the interior regularity of weak solutions of the Navier-Stokes equations},
%Arch. Rational Mech. Anal.  9, 187-195(1962).



%\bibitem{Sh}  Z. Shen, {\it Boundary value problems for parabolic Lame
%systems and a nonstationary linearized system of Navier-Stokes
%equations in Lipschitz cylinders}, Amer. J. Math. $\bf{ 113}$, no.
%2, 293-373(1991).


%\bibitem{sol-2}
%V. A. Solonnikov, {\it
% On estimates of solutions of the non-stationary Stokes problem in anisotropic Sobolev spaces and on estimates for the resolvent of the Stokes operator,} Russ. Math. Surv. 58,  no. 2, 331-365(2003).

%
%\bibitem{Sol} V.A. Solonnikov, {\it Estimates for solutions to nonstationary linearlized Navier-Stokes equations,} Tr. Mat. Inst. Steklova, 70,213-317(1964).


\bibitem{sol1}V.A. Solonnikov, {\it Estimates of the solutions of the nonstationary Navier-Stokes system,}
Boundary value problems of mathematical physics and related questions in the theory of functions, 7.
Zap. Naučn. Sem. LOMI. 38, 153-231(1973).


%
%\bibitem{sol2}V.A. Solonnikov,{\it An initial-boundary value problem for a generalized system of Stokes equations in a half-space},  Zap. Nauchn. Sem. POMI.  224-275, 271(2000);  translation in  J. Math. Sci. (N. Y.)  115,  no. 6, 2832-2861(2003).


%\bibitem{sol3}V.A. Solonnikov,{\it On nonstationary Stokes problem and Navier-Stokes problem in a half-space with initial data nondecreasing at infinity}, Function theory and applications. J. Math. Sci. (N. Y.)  114,  no. 5, 1726-1740(2003).

%\bibitem{sol4}V.A. Solonnikov ,{\it Estimates for solutions to the nonstationary Navier-Stokes equations }, Zap.Nauchn.Semin.LOMI,38,153-231(1973).
%
%\bibitem{sol5}V.A. Solonnikov,{\it Weighted Schauder estimates for evolution Stokes problem}, Ann. Univ. Ferrara Sez. VII Sci. Mat.  52,  no. 1, 137-172(2006).


\bibitem{Sol-1} V.A. Solonnikov, {\it $L_p$-estimates for solutions to the initial boundary-value problem for the generalized Stokes system in a bounded domain,} Function theory and partial differential equations. J. Math. Sci. (New York) 105, no. 5, 2448-2484(2001).

\bibitem{Sol-2} V.A. Solonnikov, {\it
Estimates for solutions of the nonstationary Stokes problem in anisotropic Sobolev spaces and estimates for the resolvent of the Stokes operator,} (Russian)  Uspekhi Mat. Nauk  58,  no. 2(350), 123-156(2003);  translation in  Russian Math. Surveys 58, no. 2, 331-365(2003).

%\bibitem{So2}
% V. A. Solonnikov, {\it On the theory of nonstationary hydrodynamic potentials.}  The Navier-Stokes equations: theory and numerical methods(Varenna, 2000),  113-129, Lecture Notes in Pure and Appl. Math., 223, Dekker, New York, 2002.
%
%\bibitem{So3} M. Sh. Birman, S. Hildebrandt, V. A. Solonnikov and N. N. Uraltseva,
%{\it Nonlinear problems in mathematical physics and related topics. I},
% International Mathematical Series (New York), 1. Kluwer Academic 2002.


\bibitem{St} E.M. Stein, {\it Singular Integrals and Differentiability Properties of Functions,}
Princeton Mathematical Series, No. 30 Princeton University Press, Princeton, N.J. 1970.

\bibitem{Tr} H. Triebel, {\it Interpolation Theory, Function Spaces, Differential Operators}, North-Holland Publishing company, 1978.


\bibitem{Triebel} H. Triebel, {\it Theory of Function Spaces,} Monographs in Mathematics, 78. Birkh$\ddot{\rm a}$user Verlag, Basel, 1983.


%\bibitem{Triebel1} H. Triebel, {\it  Theory of  Function Spaces. II,} Monographs in Mathematics, 84. Birkh$\ddot{\rm a}$user Verlag, Basel, 1992.


\bibitem{Triebel2} H. Triebel, {\it  Theory of Function Spaces. III,} Monographs in Mathematics, 100. Birkh$\ddot{\rm a}$user Verlag, Basel, 2006.


\bibitem{voss} K.A. Voss, {\it Self-similar solutions of the Navier-Stokes equation},
Thesis (Ph.D.)–Yale University. 1996.

\bibitem{ukai} S. Ukai, {\it A solution formula for the Stokes equation in $\R_+$}, Comm. in Pure and Appl. Math., XL, 611-621(1987).

\bibitem{yamazaki} M.Yamazaki, {\it A quasi-homogeneous version of paradifferential operators, I.Boundedness on spaces of Besov type,}J.Fac.Sci.Tokyo 33,131-174(1986).
%\bibitem{aman}H. Amann, {\it On the strong solvability of the Navier-Stokes equations,} J. math. fluid  mech., {\bf 2}, 16-98(2000).

%\bibitem{B} R. Brown, {\it Area integral estimates for caloric functions}, Trans. Amer. Math. Soc, {\bf 315}, no. 2, 565-589(1989).


%\bibitem{K} K. Kang, {\it On boundary regularity of the Navier-Stokes equations}, Comm. Partial. Differential Equations, {\bf 29}, no 7-8, 955-987(2004).
%
%\bibitem{kato}T. Kato, {\it Strong $L^p$-solutions of the Navier-Stokes equation in $R^m$, with applications to weak solutions,} Math. Z.187,471-480(1984).





%\bibitem{St} E. Stein, {\it Singular Integrals and Differentiability Pproperties of Functions},
%                        Princeton University Press, 1970.
%
%\bibitem{Tr} H. Triebel, {\it Interpolation Theory, Function Spaces, Differential Operators, Second edition},
%             Johann Ambrosius Barth, Heidelberg, 1995.


%\bibitem{wiegner}M. Wiegner, {\it The
%Navier-Stokes equations-a neverending challenge}, Jahresbericht DMV 101, 1-25(1999).














\end{thebibliography}
\end{document}